\documentclass[11pt,reqno]{amsart}

\usepackage{amsmath, amsfonts, amsthm, amssymb, stmaryrd, color, cite}

\usepackage{hyperref}
\newtheorem{theorem}{Theorem}[section]
\newtheorem{lemma}{Lemma}[section]
\newtheorem{proposition}{Proposition}[section]

\theoremstyle{definition}
\newtheorem{definition}{Definition}[section]

\theoremstyle{remark}

\numberwithin{equation}{section}
\allowdisplaybreaks

\def\f{\frac}

\def\hf1{^\f{1}{1-\xi^2}}

\def\be{\begin{equation}}
\def\en{\end{equation}}
\def\bs{\begin{split}}
\def\es{\end{split}}
\def\ba{\begin{align}}
\def\ea{\end{align}}

\renewcommand{\v}{{\bf v}}

\allowdisplaybreaks

\author[Z. Qiu]{Zhaoyang Qiu}
\address{School of Mathematics and Statistics, Huazhong University of Science and Technology, Wuhan, 430074, China.}
\email{ZHQMATH@163.com}

\author[Y. Tang]{Yanbin Tang}
\address{School of Mathematics and Statistics, Huazhong University of Science and Technology, Wuhan, 430074, China.}
\email{tangyb@hust.edu.cn}

\author[H. Wang]{Huaqiao Wang}
\address{College of Mathematics and Statistics, Chongqing University, Chongqing, 401331, China.}
\email{hqwang111@163.com }

\title[Asymptotic behavior for stochastic LLB equations]{Asymptotic behavior for the 1-D stochastic Landau Lifshitz Bloch equation}

\thanks{Z. Qiu is supported by the CSC under grant No. 201806160015. Y. Tang is supported by NSFC grants 11971188, 11471129. H. Wang is supported by National Postdoctoral Program for Innovative Talents (No. BX201600020) and Project No. 2019CDXYST0015 supported by the Fundamental Research Funds for the Central Universities.}


\date{\today}

\begin{document}
\begin{abstract}
The stochastic Landau-Lifshitz-Bloch equation describes the phase spins in a ferromagnetic material and has significant role in simulating heat-assisted magnetic recording. In this paper, we consider the deviation of the solution to the 1-D stochastic Landau-Lifshitz-Bloch equation, that is, we give the asymptotic behavior of the trajectory $\frac{u_\varepsilon-u_0}{\sqrt{\varepsilon}\lambda(\varepsilon)}$ as $\varepsilon\rightarrow 0+$, for $\lambda(\varepsilon)=\frac{1}{\sqrt{\varepsilon}}$ and $1$ respectively. In other words, the large deviation principle and the central limit theorem are established respectively.
\end{abstract}

\maketitle
\section{Introduction}
The Landau-Lifshitz-Gilbert (LLG) equation describes the dynamical behavior of magnetization in a ferromagnetic material below the Curie temperature $T_c$, see \cite{Gilbert,Landau}. The specific form is as follows:
\begin{eqnarray*}
dm=\lambda_1m\times H_{e}dt-\lambda_2m\times(m\times H_{e})dt,
\end{eqnarray*}
where the magnetization $m$ is in the two-dimensional sphere $\mathbb{S}^2$. However, for high temperature, we must use the Landau-Lifshitz-Bloch (LLB) equation derived by \cite{G1,G2}, to describe this phenomenon, which is actually valid for the full range of temperature. Essentially, the system consists of the LLG equation at low temperature and the Ginzburg-Landau theory of phase transitions. Let $u=\frac{m}{m_0}$ be the average spin polarization, where $m_0$ is the saturation magnetization value at $T=0$. Then, the LLB equation has the following form:
\begin{eqnarray*}
du=\gamma u\times H_{e}dt+ \frac{L_{1}}{|u|^{2}}(u\cdot H_{e})udt-\frac{L_{2}}{|u|^{2}}u\times(u\times H_{e})dt,
\end{eqnarray*}
where $H_{e}$ denotes the effect field equipped with the form,
\begin{eqnarray*}
H_{e}=\triangle u-\frac{1}{\zeta}\left(1+\frac{3}{5}\frac{T}{T-T_c}|u|^2\right)u,
\end{eqnarray*}
$\zeta$ denotes the longitudinal susceptibility. The symbol $|\cdot|$ is the Euclidean norm in $\mathbb{R}^3$, $\times$ stands for the vector cross product in $\mathbb{R}^3$, $\gamma>0$ represents the gyromagnetic ratio, and $L_1,L_2$ stand for the longitudinal and transverse damping parameters respectively.

Using the fact that $a\times (b\times c)=b(a\cdot c)-c(a\cdot b)$, we have
\begin{eqnarray*}
u\times(u\times H_{e})=(u\cdot H_{e})u-|u|^2H_{e}.
\end{eqnarray*}
We consider the case that the temperature $T>T_c$. As a result, the longitudinal $L_1$ is equal to the transverse damping parameter $L_2$, denoting $\nu_{1}:=L_1=L_2$. Therefore, the system can be rewritten as:
\begin{eqnarray*}
du=\gamma u\times H_{e}dt+\nu_1H_{e}dt,
\end{eqnarray*}
that is,
\begin{eqnarray}\label{equ1.1}
du=\nu_{1}\triangle u dt+ \gamma u\times \triangle udt-\nu_{2}(1+\mu|u|^{2})udt,
\end{eqnarray}
where $\nu_2:=\frac{\zeta}{\nu_1}$ and $\mu:=\frac{3}{5}\frac{T}{T-T_c}$.

Because of its physical importance, mathematical challenges, this model receives extensive studies and some progresses have been made in the deterministic case, in \cite{Le} for the existence of global weak solution and in \cite{Xu} for the existence of strong time periodic solution with an external magnetic field and established the time regularity in $\mathbb{R}^3$.

In the theory of ferromagnetism, describing the phase transitions disturbed by random thermal fluctuations which is significant problem and  gains lots of traction. Therefore, the stochastic factors should be taken into account in the description of the dynamics of the magnetization, to reveal the transition caused by noise. The works \cite{Banas,Z2} introduced the stochastic term into system (\ref{equ1.1}) by perturbing the effect field. That is, replacing $H_{e}$ by $H_{e}+\mathcal{B}$, $\mathcal{B}$ is white noise, which will be introduced later. Therefore, system (\ref{equ1.1}) becomes
\begin{eqnarray*}
\frac{du}{dt}=\nu_{1}\triangle u + \gamma u\times \left(\triangle u+\frac{d\mathcal{B}}{dt}\right)-\nu_{2}(1+\mu|u|^{2})u+\nu_1\frac{d\mathcal{B}}{dt}.
\end{eqnarray*}

Let $\mathcal{W}$ be a Wiener process defined on $H=L^2(\mathcal{D})$ with covariance operator $Q$, where $Q$ is a linear positive operator on  $L^2(\mathcal{D})$, which is trace and hence compact. Let $\{e_{k}\}_{k\geq 1}$ be a complete orthonormal basis of $L^2(\mathcal{D})$ such that $Qe_{k}=\lambda_{k}e_{k}$, then $\mathcal{W}$ can be written formally as the expansion $\mathcal{W}(t,\omega)=\sum_{k\geq 1}\sqrt{\lambda_{k}}e_{k}\mathcal{W}_{k}(t,\omega)$, where $\{\mathcal{W}_{k}\}$ is a sequence of independent standard real-valued 1-D Brownian motions. We also have that $\mathcal{W}\in \mathcal{C}([0,\infty),L^2(\mathcal{D}))$ almost surely, see \cite{Zabczyk}. Therefore, for each $k\in \mathbb{N}$, $\widetilde{G}_{k}:=\widetilde{G}e_k=GQ^\frac{1}{2}e_{k}$, we define $\mathcal{B}:=G\mathcal{W}=\sum_{k\geq 1}\widetilde{G}_k\mathcal{W}_k$.

Let $L_{Q}(\mathcal{H}_0,X)$ denote the collection of Hilbert-Schmidt operators, the set of all linear operators $\mathcal{K}$ such that $\mathcal{K}Q^\frac{1}{2}: H\rightarrow X$, endowed with the norm
\begin{eqnarray*}
\|\mathcal{K}\|_{L_Q(\mathcal{H}_0;X)}^2:=\sum_{k\ge1}||\mathcal{K}Q^{\frac{1}{2}}e_{k}||_{X}^{2}=\sum_{k\ge1}||[\mathcal{K}Q^{\frac{1}{2}}]^*e_{k}||_{X}^{2},
\end{eqnarray*}
where $\mathcal{H}_0=Q^\frac{1}{2}H$ and $X$ is a separable Hilbert space. Throughout the paper, we assume that $G\in L_{Q}(\mathcal{H}_0;H^1)$, hence
\begin{eqnarray}\label{equ1.2}
\sum_{k\geq 1}\|\widetilde{G}_{k}\|_{H^{1}}^2<\infty.
\end{eqnarray}

Here we just consider the stochastic LLB equation with linear noise
\begin{eqnarray}\label{equ1.3}
\frac{du}{dt}=\nu_{1}\triangle u + \gamma u\times \triangle u-\nu_{2}(1+\mu|u|^{2})u+u\times G\frac{d\mathcal{W}}{dt}.
\end{eqnarray}
For the system (\ref{equ1.3}), Jiang, Ju and Wang\cite{Jiang} established the existence of weak (in the sense of partial differential equation (PDE)) martingale solution in 3-dimensional bounded domains, and Le \cite{Le} considered the strong pathwise solution in 1- or 2-dimensional bounded domains, the martingale solution in 3-dimensional case and the existence of invariant measure. Moreover, in the case of degenerated additive noise, Guo, Huang and Wang\cite{guo} proved the uniqueness of the invariant measure for the corresponding transition semigroup.

In this paper, we are devoted to establishing the asymptotic properties of distribution of the solution $u_\varepsilon$, that is, the asymptotic behavior of the trajectory, $\frac{u_\varepsilon(t)-u_0(t)}{\sqrt{\varepsilon} \lambda(\varepsilon)}$, where $\lambda(\varepsilon)$ is the deviation scale, $u_0$ is the solution of system (\ref{equ1.1}), and $u_\varepsilon$ is the solution of following stochastic system
\begin{eqnarray}\label{equ1.4}
\frac{du_\varepsilon}{dt}=\nu_{1}\triangle u_\varepsilon + \gamma u_\varepsilon\times \triangle u_\varepsilon-\nu_{2}(1+\mu|u_\varepsilon|^{2})u_\varepsilon+\sqrt{\varepsilon}u_\varepsilon\times G\frac{d\mathcal{W}}{dt},
\end{eqnarray}
with the initial data $u_\varepsilon(0,x)=u(0)$ and $u_\varepsilon|_{\partial \mathcal{D}}=0$ for $t\in [0,\infty)$. Here we focus on the cases of $\lambda(\varepsilon)=\frac{1}{\sqrt{\varepsilon}}$ and $\lambda(\varepsilon)=1$ respectively. More precisely, we shall establish the large deviation principle (LDP) and the central limit theorem corresponding to the deviation scale $\lambda(\varepsilon)=\frac{1}{\sqrt{\varepsilon}}$ and $\lambda(\varepsilon)=1$ respectively.

In various papers LDP of solutions to stochastic partial differential equations (SPDEs) was established by the weak convergence method based on the variational representations of infinite-dimensional Wiener processes introduced by \cite{Budhiraja, Dupuis}. We refer to \cite{Millet,Sundar,XZ,ZZ} and the references therein for the 2-D Navier-Stokes equations, \cite{Duan} for the Boussinesq equations, \cite{Z1} for the LLG equation, \cite{zhang} for tamed 3D Navier-Stokes equations, \cite{Chueshov} for more general hydrodynamic models.

Although our proofs also rely on the weak convergence method, we develop some new estimates due to the complexity of the nonlinear term. Even so, only the case of $d=1$ is considered, and there exist some technique difficulties in $d=2,3$. Different from \cite{Duan,Sundar}, we shall establish the convergence of the law of $u_\varepsilon$ on space $\mathcal{C}([0,T];H^1)\cap L^2(0,T;H^2)$ via the a priori strong convergence property, $u_\varepsilon\rightarrow u_0$ in $L^2(0,T; H^{1})$ as $\varepsilon\rightarrow 0$ with the spirit of \cite{WZZ}, instead of dealing with localized integral estimates of the time increments.

Before the proof of weak convergence, we give a simplified proof of the existence and uniqueness of solution to stochastic controlled system, and obtain the uniform a priori estimates independent of $\varepsilon$ which cannot be obtained from the corresponding equations using the Girsanov transformation.

When $\lambda(\varepsilon)=1$, we shall show that $\frac{u_\varepsilon(t)-u_0(t)}{\sqrt{\varepsilon}}$ converges to a solution $V_0$ to the following system (central limit theorem):
\begin{eqnarray}\label{equ1.5}
&&dV_0-\nu_1\triangle V_0dt=\gamma(V_0\times \triangle u_0+u_0\times \triangle V_0)dt\nonumber\\
&&\qquad\qquad\qquad\quad-\nu_2(2\mu(u_{0}\cdot V_0)u_0+(1+\mu|u_0|^2)V_0)dt+u_0\times Gd\mathcal{W}.
\end{eqnarray}
We also get the well-posedness of system (\ref{equ1.5}) and the estimation of the a priori bound for the process $\frac{u_\varepsilon(t)-u_0(t)}{\sqrt{\varepsilon}}$. The high nonlinearity of terms $|u|^2u$ and $u\times \triangle u$ makes the estimates challenging.

The rest of the paper is arranged as follows. In Section \ref{sec2}, we recall some deterministic and stochastic preliminaries associated with system (\ref{equ1.3}) and then state our results. Section \ref{sec3} gives the global existence and uniqueness of solution, and the uniform a priori estimates for the controlled system. The LDP is then proved in Section \ref{sec4}. Section 5 establishes the central limit theorem.

\maketitle
\section{Preliminaries and main results}\label{sec2}

In this section, we begin by reviewing some deterministic and stochastic preliminaries \cite{Zabczyk,Ellis,Wang2014} and then give our results.

Let $\mathcal{D}\subset \mathbb{R}$ be an open bounded domain. $H^m(\mathcal{D})$ denotes the Sobolev spaces of functions having distributional derivatives up to order $m\in \mathbb{N}^+$ integrable in $L^2(\mathcal{D})$, endowed with the following norm
\begin{eqnarray}\label{equ2.1}
\|u\|_{H^m}^2:=\sum_{|\alpha|\leq m}\|\partial^\alpha u\|_{L^{2}}^2.
\end{eqnarray}
The inner product of $H^m$ will be denoted by $(\cdot,\cdot)_{H^m}=\sum_{|\alpha|\leq m}(\partial^\alpha\cdot, \partial^\alpha\cdot)$, where the symbol $(\cdot,\cdot)$ represents the inner product of $L^2(\mathcal{D})$. Due to the Dirichlet boundary condition, the well-known Poincar\'{e} inequality holds:
\begin{eqnarray*}
\|u\|_{L^2}\leq C\|\nabla u\|_{L^2}.
\end{eqnarray*}
Therefore, the norms $\|u\|_{H^1}$, $\|u\|_{H^2}$ are equivalent to the norms $\|\nabla u\|_{L^2}$, $\|\triangle u\|_{L^2}$ respectively and the following interpolation inequality holds:
\begin{eqnarray}\label{equ2.2}
\|u\|_{L^\infty(\mathcal{D})}^2\leq 2 \|u\|_{L^2(\mathcal{D})}\|u\|_{H^1(\mathcal{D})}.
\end{eqnarray}

The following estimates which will be used throughout the paper.
\begin{lemma}\label{lem2.1} If $u\in H^1$, $v,w\in H^2$, then
\begin{eqnarray}
&&(u\times v,v)=0, ~(u\times \triangle v, u)=0, \label{equ2.3}\\
&&|(u\times \triangle v, \triangle w)|\leq C_{1}\|\triangle w\|_{L^2}^2+C_2\|\triangle v\|_{L^{2}}^2\|u\|_{H^1}^2, \label{equ2.4}
\end{eqnarray}
where $C_1,C_2$ are two constants.
\end{lemma}
\begin{proof} (\ref{equ2.3}) can be obtained by a simple calculation. For (\ref{equ2.4}), by the interpolation inequality (\ref{equ2.2}) and the H\"{o}lder inequality, we have
\begin{align*}
|(u\times \triangle v, \triangle w)|&\leq\|\triangle w\|_{L^2}\|\triangle v\|_{L^{2}}\|u\|_{L^\infty}\leq C\|\triangle w\|_{L^2}\|\triangle v\|_{L^{2}}\|u\|_{H^1}\\
&\leq C_{1}\|\triangle w\|_{L^2}^2+C_2\|\triangle v\|_{L^{2}}^2\|u\|_{H^1}^2.
\end{align*}
This completes the proof.
\end{proof}

The following spaces involving fractional derivative in time are useful since the solutions of stochastic system  are H\"{o}lder continuous of order strictly less than $\frac{1}{2}$ with respect to time $t$.

For any fixed $p>1$ and $\alpha\in(0,1)$ we define,
\begin{equation*}
W^{\alpha,p}(0,T;X)=\left\{v\in L^{p}(0,T;X):\int_{0}^{T}\int_{0}^{T}\frac{\|v(t_{1})-v(t_{2})\|_{X}^{p}}{|t_{1}-t_{2}|^{1+\alpha p}}dt_{1}dt_{2}<\infty\right\},
\end{equation*}
endowed with the norm,
\begin{equation*}
\|v\|_{W^{\alpha,p}(0,T;X)}^p:=\int_{0}^{T}\|v(t)\|_{X}^{p}dt+\int_{0}^{T}\int_{0}^{T}\frac{\|v(t_{1})-v(t_{2})\|_{X}^{p}}{|t_{1}-t_{2}|^{1+\alpha p}}dt_{1}dt_{2},
\end{equation*}
where $X$ is a separable Hilbert space.

For the case $\alpha=1$, we take,
\begin{equation*}
W^{1,p}(0,T;X):=\left\{v\in L^{p}(0,T;X):\frac{dv}{dt}\in L^{p}(0,T;X)\right\},
\end{equation*}
which is the classical Sobolev space with its usual norm,
\begin{equation}\label{equ2.5}
\|v\|_{W^{1,p}(0,T;X)}^{p}:=\int_{0}^{T}||v(t)||_{X}^{p}+||v'(t)||_{X}^{p}dt.
\end{equation}
Note that for $\alpha\in(0,1)$, $ W^{1,p}(0,T;X)\subset W^{\alpha,p}(0,T;X)$.

Given an $X$-valued predictable process $f\in L^{2}(\Omega;L^{2}_{loc}([0,\infty),L_{Q}(\mathcal{H}_0,X)))$. Taking $f_{k}=fQ^\frac{1}{2}e_{k}$, one can define the stochastic integral,
\begin{eqnarray}\label{equ2.6}
M_{t}:=\int_{0}^{t}fd\mathcal{W}=\sum_{k\ge1}\int_{0}^{t}f_{k}d\mathcal{W}_{k},
\end{eqnarray}
as an element in $\mathcal{M}_{X}^{2}$ which is the space of all $X$-valued square integrable martingales \cite{Zabczyk}. For process $\{M_{t}\}_{t\geq 0}$, the Burkholder-Davis-Gundy inequality implies the following inequalities of mathematical expectation
\begin{eqnarray}\label{equ2.7}
{\mathbb E}\left(\sup_{0\leq t\leq T}\left\|\int_{0}^{t}fd\mathcal{W}\right\|_{X}^{p}\right)\leq C{\mathbb E}\left(\int_{0}^{T}\|f\|_{L_{Q}(\mathcal{H}_0;X)}^{2}dt\right)^{\frac{p}{2}},\;\forall \; p\geq1.
\end{eqnarray}
As in \cite{F1}, we also have for any $p\geq 2$ and any $\alpha \in [0,\frac{1}{2})$,
\begin{eqnarray}\label{equ2.8}
\mathbb{E}\left\|\int_{0}^{t}fd\mathcal{W}\right\|_{W^{\alpha,p}(0,T;X)}^{p}\leq C\mathbb{E}\int_{0}^{T}\left\|f\right\|^{p}_{L_{Q}(\mathcal{H}_0;X)}dt.
\end{eqnarray}
In addition, by the condition (\ref{equ1.2}), we have
\begin{eqnarray}\label{equ2.9}
\|u\times G\|_{L_Q(\mathcal{H}_0;H^1)}^2\leq C\|u\|_{H^1}^2,\; \forall\; u\in H^1.
\end{eqnarray}

In fact, according to the definition of $L_Q$ and the interpolation inequality (\ref{equ2.2}), we have
\begin{eqnarray}\label{equ2.10}
&&\|u\times G\|_{L_Q(\mathcal{H}_0;H^1)}^2=\sum_{k\geq1}\|u\times \widetilde{G}_k\|_{H^1}^2\nonumber\\
&&\leq C\sum_{k\geq1}\|\nabla u\times \widetilde{G}_k\|_{L^2}^2+\|u\times \nabla \widetilde{G}_k\|_{L^2}^2\nonumber\\
&&\leq C\sum_{k\geq1}(\|\nabla u\|_{L^2}^2\|\widetilde{G}_k\|_{L^\infty}^2+\|u\|_{L^\infty}^2\|\widetilde{G}_{k}\|_{H^1}^2)\leq C\|u\|_{H^1}^2.
\end{eqnarray}

Next, we give the main results of this paper.
\begin{theorem}\label{the2.1} Suppose that the initial data $u(0)\in L^{p}(\Omega; H^{1})$ and the condition (\ref{equ1.2}) holds. Then, for any  $\varepsilon\in(0,1]$, the solution $\{u_{\varepsilon}\}_{\varepsilon\in (0,1]}$ to system (\ref{equ1.4}) satisfies the large deviation principle on space $L^2(\Omega, L^\infty(0,T;H^1)\cap L^2(0,T;H^2))$ with good rate function
\begin{eqnarray*}
I(u)=\inf_{\{h\in L^{2}(0,T;\mathcal{H}_{0}):u=\mathcal{G}^{0}(\int_{0}^{\cdot}h(s)ds)\}}\left\{\frac{1}{2}\int_{0}^{T}\|h\|_{\mathcal{H}_{0}}^{2}dt\right\},
\end{eqnarray*}
where the infimum of empty set is taken to be infinity.
\end{theorem}
\begin{theorem}\label{the2.2}  Assume that the operator $G$ and the initial data $u(0)$ satisfy the same conditions as in Theorem 2.1. Then, the solution $\{u_{\varepsilon}\}_{\varepsilon\in (0,1]}$ to system (\ref{equ1.4}) satisfies the central limit theorem on space $L^2(\Omega, L^\infty(0,T;H^1)\cap L^2(0,T;H^2))$, that is, for any $t\in[0,T]$,
\begin{eqnarray*}
\lim_{\varepsilon\rightarrow 0}\left[\mathbb{E}\left(\sup_{s\in [0,t]}\Big\|\nabla\left(\frac{u_\varepsilon-u_0}{\sqrt{\varepsilon}}-V_0\Big)\right\|_{L^2}^2\right)+\nu_1\mathbb{E}\int_{0}^{t}\Big\|\triangle \Big(\frac{u_\varepsilon-u_0}{\sqrt{\varepsilon}}-V_0\Big)\Big\|_{L^2}^2ds\right]=0.
\end{eqnarray*}
\end{theorem}
We have reserved the details on the notation used above for Sections \ref{sec3}, \ref{sec4}, \ref{sec5}.

\maketitle
\section{Well-posedness of the stochastic system}\label{sec3}
In this section, we aim to show the LDP for the solution $u_{\varepsilon}$ of system (\ref{equ1.4}) as $\varepsilon\rightarrow 0$. We first show the existence and uniqueness of solution to the following stochastic controlled LLB equation:
\begin{eqnarray}\label{equ3.1}
du=\nu_{1}\triangle u dt+ \gamma u\times \triangle udt-\nu_{2}(1+\mu|u|^{2})udt+\sqrt{\varepsilon}u\times Gd\mathcal{W}+(u\times G)hdt,
\end{eqnarray}
where $h$ is an $\mathcal{H}_{0}$-valued predictable stochastic process satisfying $\int_{0}^{T}\|h\|_{\mathcal{H}_0}^{2}dt<\infty$, a.s. $\mathcal{H}_{0}$ is a Hilbert space defined by $\mathcal{H}_{0}:=Q^{\frac{1}{2}}H$, with the induced norm $\|\cdot\|_{\mathcal{H}_0}^{2}=\langle\cdot,\cdot\rangle_{\mathcal{H}_{0}}=(Q^{-\frac{1}{2}}\cdot,Q^{-\frac{1}{2}}\cdot)$.

For any fixed $M>0$, we define the set
\begin{eqnarray*}
S_{M}=\left\{h\in L^{2}(0,T;\mathcal{H}_{0}):\int_{0}^{T}\|h\|_{\mathcal{H}_0}^{2}dt\leq M\right\}.
\end{eqnarray*}
The set $S_{M}$ endows with the weak topology
$$
d(h,g)=\sum_{k\geq 1}\frac{1}{2^{k}}\left|\int_{0}^{T}\langle h(t)-g(t), \xi_{k}\rangle_{\mathcal{H}_{0}}dt\right|,
$$
for $g, h\in S_M$, which is a Polish space and $\{\xi_{k}\}_{k\geq 1}$ is an orthonormal basis of $L^{2}(0,T;\mathcal{H}_{0})$. For $M>0$, define $\mathcal{A}_{M}=\{h\in \mathcal{A}:h(\omega)\in S_{M}, {\rm a.s.}\}$, where $\mathcal{A}$ is the set of the process $h$.

\subsection{Global existence of the solution}\label{sec31}

By the Yamada-Watanabe argument, the strong pathwise solution follows once we show the existence of martingale solution and the uniqueness of the pathwise solution. The rigorous proof of the existence of the martingale solution bases on the Galerkin approximation, the compactness argument, and the identification of limit. Here we just state the necessary a priori estimates since the extra extend term $(u\times G)h$ and the noise coefficient $\varepsilon$ appear in system (\ref{equ3.1}). For the technical detail of the compactness argument and passing the limit, we refer to \cite{Le}, where the analysis was implemented for the original system (\ref{equ1.3}).

In general, these calculations are performed in the following Lemma on the Galerkin approximate solutions, then the estimation of $u_\varepsilon$ shall be obtained by a limiting procedure.
\begin{lemma}\label{lem3.1} Let $h\in \mathcal{A}_{M}$ for any fixed $M>0$. Assume that the initial data $u(0)\in L^{2p}(\Omega; H^{1})$ and the operator $G$ satisfies the Condition (\ref{equ1.2}), then for $\varepsilon\in (0,1]$ and $p\geq 1$, the solution $u_\varepsilon$ of system (\ref{equ1.4}) satisfies the following a priori estimates, for $t\in [0,T]$
\begin{eqnarray}\label{equ3.2}
&&\sup_{\varepsilon\in (0,1]}\mathbb{E}\left(\sup_{s\in [0,t]}\|\nabla u_\varepsilon\|_{L^{2}}^{2p}\right)+\nu_{1}\sup_{\varepsilon\in (0,1]}\mathbb{E}\int_{0}^{t}\|\nabla u_\varepsilon\|_{L^{2}}^{2p-2}\|\triangle u_\varepsilon\|_{L^{2}}^{2}ds\nonumber\\
&&+\nu_{2}\sup_{\varepsilon\in (0,1]}\mathbb{E}\int_{0}^{t}\|\nabla u_\varepsilon\|_{L^{2}}^{2p-2}(\|\nabla u_\varepsilon\|_{L^2}^{2}+\mu\||u_\varepsilon|\cdot |\nabla u_\varepsilon|\|_{L^2}^{2})ds\leq C,
\end{eqnarray}
and
\begin{eqnarray}
&&\sup_{\varepsilon\in (0,1]}\mathbb{E}\left\|u_\varepsilon-\sqrt{\varepsilon}\int_{0}^{t}u_\varepsilon\times Gd\mathcal{W}\right\|_{W^{1,2}(0,T;L^2)}^2\leq C,\label{equ3.3}\\
&&\mathbb{E}\left\|\sqrt{\varepsilon}\int_{0}^{t}u_\varepsilon\times Gd\mathcal{W}\right\|_{W^{\alpha,p}(0,T;L^2)}^p\leq C,\;\mbox{for}\; p>2, \alpha\in [0, \frac{1}{2}), \label{equ3.4}
\end{eqnarray}
where the constant $C$ is independent of $\varepsilon$ but depends on $M, T, \mathcal{D}, p$ and the initial data.
\end{lemma}
\begin{proof} For simplify the notation, we replace $u_\varepsilon$ by $u$. Using the It\^{o} formula to the function $\|\nabla u\|_{L^{2}}^{2p}$,
\begin{eqnarray}\label{equ3.5}
&&d\|\nabla u\|_{L^{2}}^{2p}+2p\nu_{1}\|\nabla u\|_{L^{2}}^{2p-2}\|\triangle u\|_{L^{2}}^{2}dt=-2p\gamma \|\nabla u\|_{L^{2}}^{2p-2}(u\times \triangle u, \triangle u)dt\nonumber\\
&&\quad+2p\nu_{2}\|\nabla u\|_{L^{2}}^{2p-2}((1+\mu|u|^{2})u, \triangle u)dt-2p\sqrt{\varepsilon}\|\nabla u\|_{L^{2}}^{2p-2}(u\times G, \triangle u)d\mathcal{W}\nonumber\\
&&\quad-2p\|\nabla u\|_{L^{2}}^{2p-2}((u\times G)h, \triangle u)dt+\varepsilon p\|\nabla u\|_{L^{2}}^{2p-2}\|u\times G\|_{L_{Q}(\mathcal{H}_0;H^1)}^{2}dt\nonumber\\
&&\quad+2p(p-1)\|\nabla u\|_{L^{2}}^{2p-4}(u\times G, \triangle u)^2dt.
\end{eqnarray}
A simple calculation gives
\begin{eqnarray}\label{equ3.6}
((1+\mu|u|^{2})u, \triangle u)=-\|\nabla u\|_{L^2}^2-3\mu(|u|^{2}, |\nabla u|^{2}),
\end{eqnarray}
and by Lemma \ref{lem2.1} (\ref{equ2.3}), we have
\begin{eqnarray}\label{equ3.7}
(u\times \triangle u, \triangle u)=0.
\end{eqnarray}
Define the stopping time $\tau_{R}$ by
\begin{eqnarray*}
\tau_{R}=\inf\left\{t\geq 0; \sup_{s\in[0,t] }\|\nabla u\|_{L^{2}}^{2p}\geq R\right\}.
\end{eqnarray*}
If the set is empty, we take $\tau_R=T$. Taking the supremum on interval $[0,t\wedge \tau_{R}]$ in (\ref{equ3.5}), and then taking expectation we have
\begin{eqnarray*}
&&\mathbb{E}\left(\sup_{s\in [0,t\wedge \tau_{R}]}\|\nabla u\|_{L^{2}}^{2p}\right)+2p\nu_{1}\mathbb{E}\int_{0}^{t\wedge \tau_{R}}\|\nabla u\|_{L^{2}}^{2p-2}\|\triangle u\|_{L^{2}}^{2}ds\nonumber\\
&&+2p\nu_{2}\mathbb{E}\int_{0}^{t\wedge \tau_{R}}\|\nabla u\|_{L^{2}}^{2p-2}(\|\nabla u\|_{L^2}^{2}+3\mu\||u|\cdot |\nabla u|\|_{L^2}^{2})ds\\
&&\leq \mathbb{E}\|\nabla u(0)\|_{L^{2}}^{2p}+2\sqrt{\varepsilon}p\mathbb{E}\left(\sup_{s\in [0,t\wedge \tau_{R})}\left|\int_{0}^{s}\|\nabla u\|_{L^{2}}^{2p-2}(u\times G, \triangle u)d\mathcal{W}\right|\right)\\
&&\quad+\varepsilon p\mathbb{E}\int_{0}^{t\wedge \tau_{R}}\|\nabla u\|_{L^{2}}^{2p-2}\|u\times G\|_{L_{Q}(\mathcal{H}_0;H^1)}^{2}ds\\
&&\quad+2p\mathbb{E}\int_{0}^{t\wedge\tau_{R}}\|\nabla u\|_{L^{2}}^{2p-2}|((u\times G)h, \triangle u)|ds\nonumber\\
&&\quad+2p(p-1)\mathbb{E}\int_{0}^{t\wedge\tau_{R}}\|\nabla u\|_{L^{2}}^{2p-4}(u\times G, \triangle u)^2ds.
\end{eqnarray*}
Regarding the stochastic term, by the Burkholder-Davis-Gundy inequality (\ref{equ2.8}) and estimate (\ref{equ2.9}),
\begin{eqnarray}\label{equ3.8}
&&2p\sqrt{\varepsilon}\mathbb{E}\left(\sup_{s\in [0,t\wedge \tau_{R}]}\left|\int_{0}^{s}\|\nabla u\|_{L^{2}}^{2p-2}(u\times G, \triangle u)d\mathcal{W}\right|\right)\nonumber\\
&&\leq p \sqrt{\varepsilon}C\mathbb{E}\left(\int_{0}^{t\wedge \tau_{R}}\|\nabla u\|_{L^{2}}^{4p-4}(u\times G, \triangle u)^{2}ds\right)^\frac{1}{2}\nonumber\\
&&\leq p\sqrt{\varepsilon}C\mathbb{E}\left(\int_{0}^{t\wedge \tau_{R}}\|\nabla u\|_{L^{2}}^{4p-2}\|u\times G\|_{L_{Q}(\mathcal{H}_0;H^1)}^{2}ds\right)^\frac{1}{2}\nonumber\\ &&\leq \frac{1}{2}\mathbb{E}\left(\sup_{s\in [0,t\wedge \tau_{R}]}\|\nabla u\|_{L^{2}}^{2p}\right)+C\varepsilon\mathbb{E}\int_{0}^{t\wedge \tau_{R}}\|\nabla u\|_{L^{2}}^{2p}ds.
\end{eqnarray}
By the H\"{o}lder inequality and (\ref{equ2.9}),
\begin{eqnarray}\label{equ3.9}
\varepsilon p\mathbb{E}\int_{0}^{t\wedge \tau_{R}}\|\nabla u\|_{L^{2}}^{2p-2}\|u\times G\|_{L_{Q}(\mathcal{H}_0;H^1)}^{2}dt\leq C\varepsilon\mathbb{E}\int^{t\wedge \tau_{R}}_{0}\|\nabla u\|_{L^{2}}^{2p}ds,
\end{eqnarray}
\begin{eqnarray}\label{equ3.10}
&&2p\mathbb{E}\int_{0}^{t\wedge \tau_{R}}\|\nabla u\|_{L^{2}}^{2p-2}((u\times G)h, \triangle u)ds\nonumber\\
&&\leq C\mathbb{E}\int_{0}^{t\wedge \tau_{R}}\|\nabla u\|_{L^{2}}^{2p-1}\|h\|_{\mathcal{H}^{0}}\|u\times G\|_{L_{Q}(\mathcal{H}_0;H^1)}ds\nonumber\\
&&\leq C\mathbb{E}\int_{0}^{t\wedge \tau_{R}}\|\nabla u\|_{L^{2}}^{2p}\|h\|_{\mathcal{H}^{0}}ds,
\end{eqnarray}
and
\begin{eqnarray}\label{equ3.11}
&&2p(p-1)\mathbb{E}\int_{0}^{t\wedge\tau_{R}}\|\nabla u\|_{L^{2}}^{2p-4}(u\times G, \triangle u)^2ds\nonumber\\
&&\leq C(p)\mathbb{E}\int_{0}^{t\wedge\tau_{R}}\|\nabla u\|_{L^{2}}^{2p-2}\|u\times G\|_{L_{Q}(\mathcal{H}_0;H^1)}^{2}ds\nonumber\\
&&\leq C(p)\mathbb{E}\int_{0}^{t\wedge\tau_{R}}\|\nabla u\|_{L^{2}}^{2p}ds.
\end{eqnarray}
Combining the estimates (\ref{equ3.8})-(\ref{equ3.11}), by the Gronwall Lemma, we have
\begin{eqnarray*}
\sup_{\varepsilon\in(0,1]}\mathbb{E}\left(\sup_{s\in [0,t\wedge \tau_{R}]}\|\nabla u\|_{L^{2}}^{2p}\right)+\nu_{1}\sup_{\varepsilon\in(0,1]}\mathbb{E}\int_{0}^{t\wedge \tau_{R}}\|\nabla u\|_{L^{2}}^{2p-2}\|\triangle u\|_{L^{2}}^{2}ds\\+\nu_{2}\sup_{\varepsilon\in(0,1]}\mathbb{E}\int_{0}^{t\wedge \tau_{R}}\|\nabla u\|_{L^{2}}^{2p-2}(\|\nabla u\|_{L^2}^{2}+3\mu\||u|\cdot |\nabla u|\|_{L^2}^{2})ds\leq C,
\end{eqnarray*}
where the constant $C$ depends on $\mathcal{D}, M,T,p$ but independent of $\varepsilon$. Then, as $R\rightarrow \infty$, by the dominated convergence theorem,
\begin{eqnarray*}
\sup_{\varepsilon\in(0,1]}\mathbb{E}\left(\sup_{s\in [0,t]}\|\nabla u\|_{L^{2}}^{2p}\right)+\nu_{1}\sup_{\varepsilon\in(0,1]}\mathbb{E}\int_{0}^{t}\|\nabla u\|_{L^{2}}^{2p-2}\|\triangle u\|_{L^{2}}^{2}ds\\+\nu_{2}\sup_{\varepsilon\in(0,1]}\mathbb{E}\int_{0}^{t}\|\nabla u\|_{L^{2}}^{2p-2}(\|\nabla u\|_{L^2}^{2}+3\mu\||u|\cdot |\nabla u|\|_{L^2}^{2})ds\leq C,
\end{eqnarray*}
where the constant $C(\mathcal{D}, M,T,p)$ is independent of $\varepsilon$.

Next, we show the inequalities (\ref{equ3.3}) and (\ref{equ3.4}). Integrating (\ref{equ3.1}) we have
\begin{eqnarray*}
&&u=u(0)+\sum_{i=1}^5J_i, \mbox{where}\;J_1=\nu_1\int_{0}^{t}\triangle u ds,~ J_2=\gamma\int_{0}^{t} u\times \triangle uds,\\
&&J_{3}=\nu_{2}\int_{0}^{t}(1+\mu|u|^{2})uds,~J_{4}=\sqrt{\varepsilon}\int_{0}^{t}u\times Gd\mathcal{W}, ~J_{5}=\int_{0}^{t}(u\times G)hds.
\end{eqnarray*}
By the estimates (\ref{equ3.2}) and the interpolation inequality (\ref{equ2.2}), we easily get
\begin{eqnarray*}
\sup_{\varepsilon\in (0,1]}\mathbb{E}\|J_{1}+J_{2}\|_{W^{1,2}(0,T;L^2)}^2\leq C,
\end{eqnarray*}
and
\begin{align*}
\sup_{\varepsilon\in (0,1]}\mathbb{E}\|J_{3}\|_{W^{1,2}(0,T;L^2)}^2&\leq C\sup_{\varepsilon\in (0,1]}\mathbb{E}\int_{0}^{T}\|(1+\mu|u|^2)u\|_{L^2}^2dt\\
&\leq C\sup_{\varepsilon\in (0,1]}\mathbb{E}\int_{0}^{T}\|u\|_{H^1}^6dt\leq C.
\end{align*}
Regarding the stochastic term, by the condition (\ref{equ1.2}), the Minkowski inequality and the Burkholder-Davis-Gundy inequality (\ref{equ2.8}), we have
\begin{align*}
\sup_{\varepsilon\in (0,1]}\mathbb{E}\|J_{4}\|_{W^{\alpha,p}(0,T;L^2)}^p&\leq C\mathbb{E}\int_{0}^{T}\|u\times G\|_{L_{Q}(\mathcal{H}_0;L^{2})}^pdt\\
&\leq C\mathbb{E}\int_{0}^{T}\left(\sum_{k\geq 1}\|u\times Ge_k\|^2_{L^2}\right)^\frac{p}{2}dt\\ &\leq C\mathbb{E}\int_{0}^{T}\int_{\mathcal{D}}\left(\sum_{k\geq 1}|u\times Ge_k|^2\right)^\frac{p}{2}dxdt\\
&\leq C\mathbb{E}\int_{0}^{T}\|u\|_{H^1}^pdt\leq C.
\end{align*}
By the fact that $h\in \mathcal{A}_M$ and the condition (\ref{equ1.2}) again, we have
\begin{align*}
\sup_{\varepsilon\in (0,1]}\mathbb{E}\|J_{5}\|_{W^{1,2}(0,T;L^2)}^2&\leq C\mathbb{E}\int_{0}^{T}\|(u\times G)h\|_{L^2}^2dt\\
&\leq C\mathbb{E}\int_{0}^{T}\|(u\times G)\|_{L_{Q}(\mathcal{H}_0;L^{2})}^2\|h\|_{\mathcal{H}_0}^2dt\\ &\leq C.
\end{align*}
This completes the proof of Lemma \ref{lem3.1}.
\end{proof}

\subsection{Uniqueness of the solution}\label{sec32}

\begin{lemma} The pathwise uniqueness of the solution holds in the following sense: suppose that $u_{1}$ and $u_{2}$ are strong pathwise solutions of system (\ref{equ3.1}). If  $\mathbb{P}\{u_{1}(0)=u_{2}(0)\}=1$, then we have
\begin{eqnarray*}
\mathbb{P}\{u_{1}(t,x)=u_{2}(t,x);\forall t\in[0,T]\}=1.
\end{eqnarray*}
\end{lemma}
\begin{proof} The difference of the solutions $v=u_{1}-u_{2}$, satisfies,
\begin{eqnarray*}
&&dv+\nu_1\triangle vdt=\gamma(-v\times \triangle u_{1}+u_{2}\times \triangle v)dt\\
&&-\nu_2(\mu(|u_{1}|^{2}-|u_{2}|^{2})u_{1}+(1+\mu|u_{2}|^{2})v)dt+\sqrt{\varepsilon}v\times Gd\mathcal{W}+(v\times G)hdt.
\end{eqnarray*}	
Using the It\^{o} formula to function $\|v\|_{L^{2}}^{2}$,
\begin{align}\label{equ3.12}
&d\|v\|_{L^{2}}^{2}+2\nu_1\|\nabla v\|_{L^{2}}^{2}dt\nonumber\\ &=-2\gamma(v\times \triangle u_{1}, v)dt+2\gamma(u_{2}\times \triangle v,  v)dt\nonumber\\
&\quad+2\nu_2\|v\|_{L^{2}}^{2}dt-2\nu_2(\mu(|u_{1}|^{2}-|u_{2}|^{2})u_{1}, v)dt-2\nu_2\mu(|u_{2}|^{2}v, v)dt\nonumber\\
&\quad+2((v\times G)h, v)dt+\varepsilon\|v\times G\|_{L_{Q}(\mathcal{H}_0;L^2)}^{2}dt+2\sqrt{\varepsilon}(v\times G, v)d \mathcal{W}\nonumber\\
&=\sum_{i=1}^{7}J_{i}dt+J_{8}d\mathcal{W}.
\end{align}
Using the fact that $(a\times \triangle b, a)=0$, we have $J_{1}=0$. By the H\"{o}lder inequality, the interpolation inequality (\ref{equ2.2}) and the fact
\begin{eqnarray*}
(u_{2}\times \triangle v,  v)=-(u_{2}\times \nabla v, \nabla v)-(\nabla u_2 \times \nabla v, v)=-(\nabla u_2 \times \nabla v, v),
\end{eqnarray*}
we have
\begin{align}
|J_{2}|&\leq  2\gamma\|v\|_{L^{\infty}}\|\nabla u_{2}\|_{L^{2}}\|\nabla v\|_{L^2}\nonumber\\ &\leq C\|v\|_{L^2}^\frac{1}{2}\|v\|_{H^1}^\frac{1}{2}\|\nabla u_{2}\|_{L^{2}}\|\nabla v\|_{L^2}\nonumber\\
&\leq \frac{\nu_1}{2}\|\nabla v\|_{L^2}^2+C\|v\|_{L^2}^2\|\nabla u_{2}\|_{L^{2}}^4,\label{equ3.13}\\
|J_{4}+J_{5}|&\leq C(|v|(|u_{1}|+|u_{2}|)u_{1},v)+\|v\|_{L^{2}}^{2}\|u_{2}\|_{L^{\infty}}^{2} \nonumber\\
&\leq C\|v\|_{L^{2}}^{2}(\|u_{1}\|_{L^{\infty}}^{2}+\|u_{2}\|_{L^{\infty}}^{2})\nonumber\\
&\leq C\|v\|_{L^{2}}^{2}\|u_{1},u_{2}\|_{H^{1}}^{2}.\label{equ3.14}
\end{align}
Using (\ref{equ2.9}) and the H\"{o}lder inequality, we obtain
\begin{eqnarray}\label{equ3.15}
|J_{6}|\leq \|v\|_{L^{2}}\|v\times G\|_{L_{Q}(\mathcal{H}_0;L^2)}\|h\|_{\mathcal{H}_{0}}\leq C\|v\|_{L^{2}}^{2}\|h\|_{\mathcal{H}^{0}}.
\end{eqnarray}
The condition (\ref{equ2.2}) yields
\begin{align}\label{equ3.16}
|J_7|&= \varepsilon\sum_{k\geq 1}\|v\times \widetilde{G}_k\|^2_{L^2}=\varepsilon\sum_{k\geq 1}\int_{\mathcal{D}}|v\times \widetilde{G}_k|^2dx\nonumber\\
&\leq C\varepsilon\sum_{k\geq 1}\| \widetilde{G}_k\|_{H^1}^2\|v\|_{L^2}^2\leq C\varepsilon\|v\|_{L^2}^2.
\end{align}
Let
\begin{eqnarray*}
&&\Psi(t)=\|v\|_{L^{2}}^{2},~\Phi(t)=\|\nabla v\|_{L^{2}}^{2},\\
&&\Lambda(t)=C(\varepsilon+1+\|h\|_{\mathcal{H}_{0}}+\|u_{1},u_{2}\|_{H^{1}}^{2}+\|u_{2}\|_{H^1}^4),~~
\varphi(t)={\rm exp}\left(-\int_{0}^{t}\Lambda(s)ds\right).
\end{eqnarray*}
Using the It\^{o} product formula to the process $\varphi(t)\Psi(t)$, combining (\ref{equ3.12})-(\ref{equ3.16}), we obtain
\begin{eqnarray}\label{equ3.17}
&&\varphi(t)\Psi(t)+\int_{0}^{t}\varphi(s)\Phi(s)ds\nonumber\\
&&\leq C\int_{0}^{t}\varphi(s)\Psi(s)ds+2\sqrt{\varepsilon}\int_{0}^{t}\varphi(s)(v\times G, v)d\mathcal{W}.
\end{eqnarray}
Taking expectation on both sides of (\ref{equ3.17}), we have
\begin{eqnarray*}
\mathbb E\varphi(t)\Psi(t)\leq C\mathbb E\int_{0}^{t}\varphi(s)\Psi(s)ds,~ t\in[0,T].
\end{eqnarray*}
Here, we use the fact that the stochastic term is a square integral martingale which its expectation vanishes. The Gronwall Lemma gives $\Psi(t)=0$, $\mathbb{P}$-a.s. for all $t\in[0,T]$. This completes the proof of uniqueness.
\end{proof}

\section{Large deviation principle}\label{sec4}

We shall establish the large deviation principle using a weak convergence approach \cite{Budhiraja,Dupuis}, based on the variational representations of infinite-dimensional Wiener processes.

For a Polish space $\mathcal{X}$, a function $I:\mathcal{X}\rightarrow [0,\infty]$ is called a rate function if $I$ is lower semicontinuous and is referred to as a good rate function if for each $M<\infty$, the level set $\{x\in \mathcal{X}:I(x)\leq M\}$ is compact. For completeness we now give the definition of large deviation and Laplace principles. For more background in this area of study we refer to \cite{Ellis}.

\begin{definition}\label{def4.1}[Large Deviation Principle] The family $\left\{U^{\varepsilon} \right\}_{\varepsilon>0}$ satisfies the LDP on $\mathcal{X}$ with rate function $I$ if the following two conditions hold,

a. LDP lower bound: for every open set $O\subset \mathcal{X}$,
\begin{equation*}
-\inf_{x\in O} I(x) \leq \liminf_{\varepsilon \rightarrow 0}\varepsilon \log \mathbb P(U^{\varepsilon} \in O);
\end{equation*}

b. LDP upper bound: for every closed set $C \subset \mathcal{X}$,
\begin{equation*}
\limsup_{\varepsilon \rightarrow 0} \varepsilon \log \mathbb P(U^{\varepsilon} \in C) \leq -\inf_{x\in C}I(x).
\end{equation*}
\end{definition}

\begin{definition}\label{def4.2}[Laplace Principle] Let $I$ be a rate function on space $\mathcal{X}$. A family $\{U^{\varepsilon}\}_{\varepsilon>0}$ of $\mathcal{X}$-valued random processes is said to satisfy the Laplace principle on $\mathcal{X}$ with rate function $I$ if for each real-valued, bounded and continuous function $f$, we have
\begin{eqnarray*}
\lim_{\varepsilon\rightarrow 0}\varepsilon\log \mathbb E\bigg\{{\rm exp}\bigg[-\frac{1}{\varepsilon}f(U^{\varepsilon})\bigg]\bigg\}=-\inf_{x\in \mathcal{X}}\{f(x)+I(x)\}.
\end{eqnarray*}
\end{definition}
Since the family $\{U^{\varepsilon}\}$ is a Polish space valued random process, the Laplace principle and the large deviation principle are equivalent, see \cite[Theorem 1.2.3]{Ellis}. To apply the weak convergence approach, we will use the following theorem given in \cite{Dupuis}. For examples of results on large deviations for stochastic PDEs by applying the theorem below see \cite{Millet,Chueshov,Duan,Sundar}.
\begin{theorem}\rm\!\!\cite[Theorem 6]{Dupuis} \label{the4.1} For Polish spaces $\mathcal{X},\mathcal{Y}$ and each $\varepsilon>0$, let $\mathcal{G}^{\varepsilon}:\mathcal{Y}\rightarrow \mathcal{X}$ be a measurable map and define $U^{\varepsilon}:=\mathcal{G}^{\varepsilon}(\sqrt{\varepsilon}\mathcal{W})$ where $\mathcal{W}$ is a $Q$-wiener process. If there is a measurable map $\mathcal{G}^{0}:\mathcal{Y}\rightarrow \mathcal{X}$ such that the following conditions hold,

$(1)$ For $M<\infty$, if $h_{\varepsilon}$ converges in distribution to $h$ as $S_{M}$-valued random elements, then,
\begin{eqnarray*}
\mathcal{G}^{\varepsilon}\left(\sqrt{\varepsilon}\mathcal{W}+\int_{0}^{\cdot}h_{\varepsilon}(s)ds\right)\rightarrow \mathcal{G}^{0}\left(\int_{0}^{\cdot}hds\right)
\end{eqnarray*}
as $\varepsilon\rightarrow 0$ in distribution $\mathcal{X}$;

$(2)$ For every $M<\infty$, the set
\begin{eqnarray*}
K_{M}=\{U_{h}^{0}:h\in S_{M}\}
\end{eqnarray*}
is a compact subset of $\mathcal{X}$.
Then, the family $\{U^{\varepsilon}\}$ satisfies the large deviation principle with the rate function
\begin{eqnarray*}
I(U)=\inf_{\{h\in L^{2}(0,T;\mathcal{H}_{0}):U=\mathcal{G}^{0}(\int_{0}^{\cdot}h(s)ds)\}}\left\{\frac{1}{2}\int_{0}^{T}\|h\|_{0}^{2}ds\right\}.
\end{eqnarray*}
\end{theorem}

The solution to the stochastic system (\ref{equ3.1}) is denoted as $u_{\varepsilon}=\mathcal{G}^{\varepsilon}(\sqrt{\varepsilon}\mathcal{W})$ for a Borel measurable function $\mathcal{G}^{\varepsilon}:\mathcal{C}([0,T];\mathcal{H}_{0})\rightarrow X$. Here $X$ is the Banach space $L^\infty(0,T;H^1)\cap L^2(0,T;H^2)$. Let $\{h_{\varepsilon}\}_{\varepsilon\in(0,1]}\subset \mathcal{A}_{M}$ be a family of random elements. Let $u_{h_{\varepsilon}}^{\varepsilon}$ be the solution of the following stochastic controlled system:
\begin{eqnarray}\label{equ4.1}
\left\{\begin{array}{ll}
du_{h_{\varepsilon}}^{\varepsilon}=\nu_{1}\triangle u_{h_{\varepsilon}}^{\varepsilon} dt+ \gamma u_{h_{\varepsilon}}^{\varepsilon}\times \triangle u_{h_{\varepsilon}}^{\varepsilon}dt\\ \qquad\qquad-\nu_{2}(1+\mu|u_{h_{\varepsilon}}^{\varepsilon}|^{2})u_{h_{\varepsilon}}^{\varepsilon} dt
+\sqrt{\varepsilon}u_{h_{\varepsilon}}^{\varepsilon}\times Gd\mathcal{W}+(u_{h_{\varepsilon}}^{\varepsilon}\times G)h_{\varepsilon}dt,\\
u_{h_{\varepsilon}}^{\varepsilon}(0)=u(0).
\end{array}\right.
\end{eqnarray}
Thanks to the uniqueness of solution to system (\ref{equ4.1}), we have
\begin{eqnarray*}
u_{h_{\varepsilon}}^{\varepsilon}=\mathcal{G}^{\varepsilon}\left(\sqrt{\varepsilon}\mathcal{W}+\int_{0}^{\cdot}h_{\varepsilon}(s)ds\right).
\end{eqnarray*}

Let $h\in \mathcal{A}_{M}$ and $u_{h}$ be the solution of the corresponding deterministic controlled system:
\begin{eqnarray}\label{equ4.2}
\left\{\begin{array}{ll}
du_{h}=\nu_{1}\triangle u_{h}dt+ \gamma u_{h}\times \triangle u_{h}dt-\nu_{2}(1+\mu|u_{h}|^{2})u_{h}dt+(u_{h}\times G)hdt,\\
u_{h}(0)=u(0).
\end{array}\right.
\end{eqnarray}
 Let $\mathcal{D}=\left\{\int_{0}^{\cdot}h(s)ds:h\in L^{2}(0,T;\mathcal{H}_{0})\right\}\subset \mathcal{C}([0,T];\mathcal{H}_{0})$ and we define the measurable map $\mathcal{G}^{0}: \mathcal{C}([0,T];\mathcal{H}_{0})\rightarrow X$ by $\mathcal{G}^{0}(g)=u_{h}$, where $g=\int_{0}^{\cdot}h(s)ds\in \mathcal{D}$, and otherwise, set $\mathcal{G}^{0}(g)=0$. Next, we establish the weak convergence of the family $\{u_{h_{\varepsilon}}^{\varepsilon}\}$ as $\varepsilon\rightarrow 0$ needed by Theorem \ref{lem4.1}.
\begin{lemma}\label{lem4.1}Let the initial data $u(0)\in L^{p}(\Omega; H^{1})$ be $\mathcal{F}_{0}$-measurable random variable and the operator $G$ satisfies Condition (\ref{equ1.2}). For every $M<\infty$, assume that $h_{\varepsilon}$ converges to $h$ in distribution as random elements taking values in $\mathcal{A}_{M}$. Then the process $\mathcal G^{\varepsilon}\left(\sqrt{\varepsilon}\mathcal{W}+\int_{0}^{\cdot}h_{\varepsilon}(s)ds\right)$ converges in distribution to $\mathcal G^{0}\left(\int_{0}^{\cdot}hds\right)$ in $X$ as $\varepsilon\rightarrow 0$, that is, the solution $u_{h_{\varepsilon}}^{\varepsilon}$ of system (\ref{equ4.1}) converges in distribution in $X$ to the solution $u_{h}$ of system (\ref{equ4.2}) as $\varepsilon\rightarrow 0$.
\end{lemma}
\begin{proof} Here, we prove directly $u_{h_{\varepsilon}}^{\varepsilon}$ converges to $u_{h}$ in probability. Let $V_{\varepsilon}=u_{h_{\varepsilon}}^{\varepsilon}-u_{h}$ be the difference of two solutions and satisfies
\begin{eqnarray*}
&&dV_{\varepsilon}-\nu_{1}\triangle V_{\varepsilon}dt=\gamma (V_{\varepsilon}\times \triangle u_{h_{\varepsilon}}^{\varepsilon}+u_{h}\times \triangle V_{\varepsilon})dt-\nu_{2}\mu(|u_{h_{\varepsilon}}^{\varepsilon}|^{2}-|u_{h}|^{2})u_{h_{\varepsilon}}^{\varepsilon}dt\nonumber\\
&&\quad -\nu_{2}(1+\mu|u_{h}|^{2})V_{\varepsilon}dt+\sqrt{\varepsilon}u_{h_{\varepsilon}}^{\varepsilon}\times Gd\mathcal{W}+((u_{h_{\varepsilon}}^{\varepsilon}\times G)h_{\varepsilon}-(u_{h}\times G)h)dt.
\end{eqnarray*}
Using the It\^{o} formula to the function $\|\nabla V_{\varepsilon}\|_{L^{2}}^{2}$, we have
\begin{align}\label{equ4.3}
d\|\nabla V_{\varepsilon}\|_{L^{2}}^{2}&+2\nu_{1}\|\triangle V_{\varepsilon}\|_{L^{2}}^{2}dt=-2\gamma(V_{\varepsilon}\times \triangle u_{h_{\varepsilon}}^{\varepsilon}+u_{h}\times \triangle V_{\varepsilon}, \triangle V_{\varepsilon})dt\nonumber\\ &+2\nu_{2}\mu((|u_{h_{\varepsilon}}^{\varepsilon}|^{2}-|u_{h}|^{2})u_{h_{\varepsilon}}^{\varepsilon}, \triangle V_{\varepsilon})dt\nonumber\\ &+2\nu_{2}((1+\mu|u_{h}|^{2})V_{\varepsilon},\triangle V_{\varepsilon})dt-2\sqrt{\varepsilon}(u_{h_{\varepsilon}}^{\varepsilon}\times G,\triangle V_{\varepsilon})d\mathcal{W}\nonumber\\
&+\varepsilon\|u_{h_{\varepsilon}}^{\varepsilon}\times G\|_{L_{Q}(H;H^1)} ^{2}dt-2((u_{h_{\varepsilon}}^{\varepsilon}\times G)h_{\varepsilon}-(u_{h}\times G)h, \triangle V_{\varepsilon})dt.
\end{align}
We next estimate all of the terms on the right hand side of (\ref{equ4.3}). Lemma \ref{lem2.1} (\ref{equ2.4}) gives $(u_{h}\times \triangle V_{\varepsilon}, \triangle V_{\varepsilon})=0$. By the H\"{o}lder inequality and Sobolev embedding $H^{1}\subset L^{\infty}$, we have
\begin{align}\label{equ4.4}
|(V_{\varepsilon}\times \triangle u_{h_{\varepsilon}}^{\varepsilon},\triangle V_{\varepsilon})|
&\leq \|\triangle V_{\varepsilon}\|_{L^{2}}\| \triangle u_{h_{\varepsilon}}^{\varepsilon}\|_{L^{2}}\|V_{\varepsilon}\|_{L^{\infty}}\nonumber\\
&\leq C\|\triangle V_{\varepsilon}\|_{L^{2}}\| \triangle u_{h_{\varepsilon}}^{\varepsilon}\|_{L^{2}}\|V_{\varepsilon}\|_{H^{1}}\nonumber\\
&\leq C\| \triangle u_{h_{\varepsilon}}^{\varepsilon}\|_{L^{2}}^{2}\|V_{\varepsilon}\|_{H^{1}}^{2}+\frac{\nu_{1}}{2}\|\triangle V_{\varepsilon}\|_{L^{2}}^{2},
\end{align}
and
\begin{align}\label{equ4.5}
|((1+\mu|u_{h}|^{2})V_{\varepsilon},\triangle V_{\varepsilon})|&\leq C\|\triangle V_{\varepsilon}\|_{L^{2}}\|V_{\varepsilon}\|_{H^{1}}(1+\mu\|u_{h}\|_{L^{\infty}}^{2})\nonumber\\
&\leq \frac{\nu_{1}}{2}\|\triangle V_{\varepsilon}\|_{L^{2}}^{2}+C\|V_{\varepsilon}\|_{H^{1}}^{2}(1+\|u_{h}\|_{H^{1}}^{4}).
\end{align}
Using the same estimate as in (\ref{equ3.14}), we gain
\begin{align}\label{equ4.6}
|(\mu(|u_{h_{\varepsilon}}^{\varepsilon}|^{2}-|u_{h}|^{2})u_{h_{\varepsilon}}^{\varepsilon}, \triangle V_{\varepsilon})|
&\leq C\|\triangle V_{\varepsilon}\|_{L^{2}}\|V_{\varepsilon}\|_{H^{1}}\|(u_{h_{\varepsilon}}^{\varepsilon},u_{h})\|_{L^{\infty}}^{2}\nonumber\\
&\leq \frac{\nu_{1}}{2}\|\triangle V_{\varepsilon}\|_{L^{2}}^{2}+C\|V_{\varepsilon}\|_{H^{1}}^{2}\|(u_{h_{\varepsilon}}^{\varepsilon},u_{h})\|_{H^{1}}^{4}.
\end{align}
The following term may be treated by same way as ({\ref{equ3.13}}),
\begin{eqnarray}\label{equ4.7}
&&((u_{h_{\varepsilon}}^{\varepsilon}\times G)h_{\varepsilon}-(u_{h}\times G)h, \triangle V_{\varepsilon})\nonumber\\
&&\leq |((V_\varepsilon\times G)h_{\varepsilon},\triangle V_{\varepsilon})|+((u_{h}\times G)(h_{\varepsilon}-h),\triangle V_{\varepsilon} )\nonumber\\
&&\leq C\|\nabla V_{\varepsilon}\|_{L^{2}}\|h_{\varepsilon}\|_{\mathcal{H}_{0}}\|V_\varepsilon \times G\|_{L_Q(\mathcal{H}_0;H^1)}+((u_{h}\times G)(h_{\varepsilon}-h),\triangle V_{\varepsilon} )\nonumber\\
&&\leq C\|\nabla V_{\varepsilon}\|_{L^{2}}^{2}\|h_{\varepsilon}\|_{\mathcal{H}_{0}}+((u_{h}\times G)(h_{\varepsilon}-h),\triangle V_{\varepsilon} ).
\end{eqnarray}
For fixed $N>0$ and $t\in [0,T]$, let
\begin{eqnarray*}
&&B_{N}(t)=\left\{w:\sup_{s\in[0,t]}\|\nabla u_{h}\|_{L^{2}}^{2}+\int_{0}^{t}\|\triangle u_{h}\|_{L^{2}}^{2}ds\leq N\right\},\\
&&B_{N}^{\varepsilon}(t)=G_{N}(t)\cap\left\{w:\sup_{s\in[0,t]}\|\nabla u_{h_{\varepsilon}}^{\varepsilon}\|_{L^{2}}^{2}+\int_{0}^{t}\|\triangle u_{h_{\varepsilon}}^{\varepsilon}\|^{2}ds\leq N\right\}.
\end{eqnarray*}
{\bf Claim 1.} We have $\sup_{\varepsilon\in (0,1]}\sup_{h, h_{\varepsilon}\in \mathcal{A}_{M}}\mathbb{P}(B_{N}^{\varepsilon}(t)^{c})\rightarrow 0~{\rm as}~N\rightarrow\infty$.

In fact, for any $h, h_{\varepsilon}\in \mathcal{A}_{M}$, it follows from Markov's inequality and energy estimates (\ref{equ3.2}) that
\begin{align*}
\mathbb{P}(B_{N}^{\varepsilon}(t)^{c})&\leq \mathbb{P}\left(\sup_{s\in[0,t]}\|\nabla u_{h}\|_{L^{2}}^{2}> N\right)+\mathbb{P}\left(\int_{0}^{t}\|\triangle u_{h}\|_{L^{2}}^{2}ds> N\right)\\
&\quad+\mathbb{P}\left(\sup_{s\in[0,t]}\|\nabla u_{h_{\varepsilon}}^{\varepsilon}|_{L^{2}}^{2}> N\right)+\mathbb{P}\left(\int_{0}^{t}\|\triangle u_{h_{\varepsilon}}^{\varepsilon}\|_{L^{2}}^{2}ds> N\right)\\
&\leq \frac{1}{N}\sup_{\varepsilon\in(0, 1],h,h_{\varepsilon}\in \mathcal{A}_{M}}\mathbb E\left(\sup_{s\in[0,t]}\|\nabla u_{h_{\varepsilon}}^{\varepsilon},\nabla u_{h}\|_{L^{2}}^{2}+\int_{0}^{t}\|\triangle u_{h_{\varepsilon}}^{\varepsilon},\triangle u_{h}\|_{L^{2}}^{2}ds\right)\\
&\leq \frac{C}{N},
\end{align*}
where the constant $C$ is independent of $N$.

Next, we show that
\begin{eqnarray}\label{equ4.8}
\mathbb E\left[{\rm I}_{B_{N}^{\varepsilon}(T)}\left(\sup_{t\in[0,T]}\|\nabla V_{\varepsilon}(t)\|_{L^{2}}^{2}+\nu_{1}\int_{0}^{T}\|\triangle V_{\varepsilon}(t)\|_{L^{2}}^{2}dt\right)\right]\rightarrow 0~ {\rm as }~ \varepsilon\rightarrow 0,
\end{eqnarray}
where the notation ${\rm I}_{\cdot}$ denotes the characteristic function. Taking into account (\ref{equ4.4})-(\ref{equ4.7}), we obtain
\begin{eqnarray}\label{equ4.9}
&&{\rm I}_{B_{N}^{\varepsilon}(T)}\left(\sup_{t\in [0,T]}\|\nabla V_{\varepsilon}\|_{L^{2}}^{2}+\nu_{1}\int_{0}^{T}\|\triangle V_{\varepsilon}\|_{L^{2}}^{2}dt\right)\nonumber\\&& \leq C\int_{0}^{T}{\rm I}_{B_{N}^{\varepsilon}(T)}(1+\|\triangle u_{h_{\varepsilon}}^{\varepsilon}\|_{L^{2}}^{2}+\|(u_{h_{\varepsilon}}^{\varepsilon},u_{h})\|_{H^{1}}^{4})\|V_{\varepsilon}\|_{H^{1}}^{2}dt\nonumber\\
&&\quad+2\sqrt{\varepsilon}{\rm I}_{B_{N}^{\varepsilon}(T)}\sup_{t\in [0,T]}\left|\int_{0}^{t}(u_{h_{\varepsilon}}^{\varepsilon}\times G,\triangle V_{\varepsilon})d\mathcal{W}\right|\nonumber\\&&\quad+2{\rm I}_{B_{N}^{\varepsilon}(T)}\int_{0}^{T}((u_{h}\times G)(h_{\varepsilon}-h),\triangle V_{\varepsilon} )dt.
\end{eqnarray}
Taking expectation on both sides of (\ref{equ4.9}), using the Gronwall Lemma, to obtain
\begin{eqnarray}\label{equ4.10}
&&\mathbb E\left[{\rm I}_{B_{N}^{\varepsilon}(T)}\left(\sup_{t\in[0,T]}\|\nabla V_{\varepsilon}(t)\|_{L^{2}}^{2}
+\nu_1\int_{0}^{T}\|\triangle V_{\varepsilon}(t)\|_{L^{2}}^{2}dt\right)\right]\nonumber\\
&&\leq e^{C(M,T,N)}\sqrt{\varepsilon}\mathbb{E}\left({\rm I}_{B_{N}^{\varepsilon}(T)}\sup_{t\in [0,T]}\left|\int_{0}^{t}(u_{h_{\varepsilon}}^{\varepsilon}\times G,\triangle V_{\varepsilon})d\mathcal{W}\right|\right)\nonumber\\&&\quad+2\mathbb{E}\left({\rm I}_{B_{N}^{\varepsilon}(T)}\left|\int_{0}^{T}((u_{h}\times G)(h_{\varepsilon}-h),\triangle V_{\varepsilon} )dt\right|\right).
\end{eqnarray}
For the first term on the right hand side of (\ref{equ4.10}), by the Burkholder-Davis-Gundy inequality (\ref{equ2.7}) and (\ref{equ2.9}),
\begin{eqnarray}\label{equ4.11}
&&\sqrt{\varepsilon}\mathbb{E}\left({\rm I}_{B_{N}^{\varepsilon}(T)}\sup_{t\in [0,T]}\left|\int_{0}^{t}(u_{h_{\varepsilon}}^{\varepsilon}\times G,\triangle V_{\varepsilon})d\mathcal{W}\right|\right)\nonumber\\
&&=\sqrt{\varepsilon}\mathbb{E}\left({\rm I}_{B_{N}^{\varepsilon}(T)}\sup_{t\in [0,T]}\left|\int_{0}^{t}(\nabla(u_{h_{\varepsilon}}^{\varepsilon}\times G),\nabla V_{\varepsilon})d\mathcal{W}\right|\right)\nonumber\\
&&\leq C\sqrt{\varepsilon}\mathbb{E}\left({\rm I}_{B_{N}^{\varepsilon}(T)}\int_{0}^{T}(\nabla (u_{h_{\varepsilon}}^{\varepsilon}\times G),\nabla V_{\varepsilon})^{2}dt\right)^{\frac{1}{2}}\nonumber\\
&&\leq C\sqrt{\varepsilon}\mathbb{E}\left({\rm I}_{B_{N}^{\varepsilon}(T)}\int_{0}^{T}\|V_{\varepsilon}\|_{H^{1}}^{2}\|u_{h_{\varepsilon}}^{\varepsilon}\times G\|_{L_Q(\mathcal{H}_0;H^{1})}^{2}dt\right)^{\frac{1}{2}}\nonumber\\
&&\leq C\sqrt{\varepsilon}\mathbb{E}\left({\rm I}_{B_{N}^{\varepsilon}(T)}\int_{0}^{T}\|V_{\varepsilon}\|_{H^{1}}^{2}\|u_{h_{\varepsilon}}^{\varepsilon}\|_{H^{1}}^{2}dt\right)^{\frac{1}{2}}\leq CN\sqrt{\varepsilon T}.
\end{eqnarray}
{\bf Claim 2.} Suppose that the condition (\ref{equ1.2}) holds, the following convergence hold,
\begin{eqnarray*}
\lim_{\varepsilon\rightarrow 0}\left({\rm I}_{B_{N}^{\varepsilon}(T)}\sup_{t\in [0,T]}\left|\int_{0}^{t}((u_{h}\times G)(h_{\varepsilon}-h),\triangle V_{\varepsilon})ds\right|\right)= 0, ~\mathbb{P}\mbox{-a.s.}
\end{eqnarray*}

In fact, by the a priori estimates (\ref{equ3.2})-(\ref{equ3.4}) and the Aubin-Lions compact embedding lemma, we can show the law of $u_{h_{\varepsilon}}^\varepsilon$ is tight on path space $L^2(0,T; H^1)$, then the Skorokhod representation theorem gives the convergence of $u_{h_{\varepsilon}}^\varepsilon$ itself on path space $L^2(0,T; H^1)$ on a new probability space $(\tilde{\Omega}, \tilde{\mathcal{F}}, \tilde{\mathbb{P}})$. Thanks to the uniqueness, we use Gy\"{o}ngy-Krylov's lemma to recover the convergence almost surely of the sequences $u_{h_{\varepsilon}}^\varepsilon$ on the original probability space. We may infer that there exists a process $u_h\in L^{2}(0,T; H^{1})$ such that $u_{h_{\varepsilon}}^\varepsilon\rightarrow u_h$ in $L^{2}(0,T; H^{1})$ $\mathbb{P}$-a.s. Finally, we need to show that the $u_h$ is a solution of system (\ref{equ4.2}). Here, the condition $h_{\varepsilon}$ converging to $h$ in distribution shall be used to identify the limit.   Following the idea of \cite{Millet}, observe that for $\phi\in H^1$ and $A\subset[0,T]$,
\begin{eqnarray*}
&&~(u_{h_\varepsilon}^\varepsilon,{\rm I}_A\phi)-\bigg(\nu_1\int_0^t(\triangle u_{h},{\rm I}_A \phi)ds+\gamma\int_0^t(u_h\times\triangle u_h, {\rm I}_A\phi)ds\\
&&\qquad\qquad\quad-\nu_2\int_{0}^{t}((1+\mu|u_h|^2)u_h,{\rm I}_A\phi)ds+\int_0^t(u_h\times G h, {\rm I}_A\phi )ds\bigg)=\sum_{i=1}^{5}J_i,
\end{eqnarray*}
where
\begin{eqnarray*}
&&J_1=\nu_1\int_0^t(\triangle u_{h_{\varepsilon}}^\varepsilon-\triangle u_h,{\rm I}_A\phi)ds,~~J_{2}=\gamma\int_0^t(u_{h_{\varepsilon}}\times \triangle u_{h_{\varepsilon}}^\varepsilon-u_h\times \triangle u_h,{\rm I}_A\phi)ds,\\
&&J_{3}=\nu_2\int_0^t((1+\mu|u_{h_{\varepsilon}}^\varepsilon|^2u_{h_{\varepsilon}}^\varepsilon)-(1+\mu|u_{h}|^2u_h),{\rm I}_A\phi)ds,\\
&&J_4=\sqrt{\varepsilon}\int_0^t(u_{h_{\varepsilon}}\times G,{\rm I}_A\phi)d\mathcal{W},~~J_5=\int_0^t(u_{h_{\varepsilon}}^\varepsilon\times G h_{\varepsilon}-u_{h}\times G h,{\rm I}_A\phi)ds.
\end{eqnarray*}
We show that all the terms $\mathbb{E}|J_i|$ converge to 0  as $\varepsilon\rightarrow 0$. Using the H\"{o}lder inequality and the interpolation inequality (\ref{equ2.2}),
\begin{eqnarray}\label{equ4.12}
\mathbb{E}|J_1|\leq \sqrt{t}\|\phi\|_{H^1}\mathbb{E}\int_0^t\|\nabla (u_{h_{\varepsilon}}^\varepsilon- u_h) \|_{L^2}^2ds\leq C\mathbb{E}\int_0^t\|\nabla (u_{h_{\varepsilon}}^\varepsilon- u_h) \|_{L^2}^2ds,
\end{eqnarray}
as well as
\begin{align}\label{equ4.13}
\mathbb{E}|J_2|&\leq\mathbb{E}\int_0^t\|\nabla (u_{h_{\varepsilon}}^\varepsilon- u_h) \|_{L^2}\|\nabla u_{h_{\varepsilon}}^\varepsilon,\nabla u_h\|_{L^2}\|\phi\|_{L^\infty}\nonumber\\&\qquad\qquad+\|u_{h_{\varepsilon}}^\varepsilon- u_h\|_{L^\infty}\|\nabla u_{h_{\varepsilon}}^\varepsilon\|_{L^2}\|\nabla \phi\|_{L^2}ds\nonumber\\
&\leq \|\phi\|_{H^1}\mathbb{E}\left(\sup_{t\in [0,T]}\|\nabla u_{h_{\varepsilon}}^\varepsilon,\nabla u_h\|_{L^2}\cdot\int_0^t\|\nabla (u_{h_{\varepsilon}}^\varepsilon- u_h) \|_{L^2}ds\right)\nonumber\\
&\leq C\mathbb{E}\int_0^t\|\nabla (u_{h_{\varepsilon}}^\varepsilon- u_h) \|_{L^2}^2ds.
\end{align}
The term $J_3$ also can be controlled as
\begin{align}\label{equ4.14}
\mathbb{E}|J_3|&\leq \|\phi\|_{L^2}\mathbb{E}\int_0^t\|\nabla (u_{h_{\varepsilon}}^\varepsilon- u_h) \|_{L^2}\|u_{h_{\varepsilon}}^\varepsilon, u_h\|_{H^1}^2ds\nonumber\\
&\leq  C\mathbb{E}\int_0^t\|\nabla (u_{h_{\varepsilon}}^\varepsilon- u_h) \|_{L^2}^2ds.
\end{align}
By the It\^{o} isometry formula and (\ref{equ2.9}),
\begin{align}\label{equ4.15}
\mathbb{E}|J_4|&\leq\sqrt{\varepsilon}\mathbb{E}\left(\int_0^t(u_{h_{\varepsilon}}\times G,{\rm I}_A\phi)^2ds\right)^\frac{1}{2}\nonumber\\
&\leq \sqrt{\varepsilon}\|\phi\|_{L^2}\mathbb{E}\left(\int_0^t\|u_{h_{\varepsilon}}\times G\|_{L_2(\mathcal{H}_0;L^2)}^2ds\right)^\frac{1}{2}\leq C\sqrt{\varepsilon}.
\end{align}
Decompose term $\mathbb{E}|J_5|$ we have
\begin{align}\label{equ4.16}
\mathbb{E}|J_5|&=\mathbb{E}\left|\int_0^t((u_{h_{\varepsilon}}^\varepsilon-u_h)\times G h_{\varepsilon}, {\rm I}_A\phi)ds+
\int_0^t ((u_h\times G)(h_\varepsilon-h),{\rm I}_A\phi)ds\right|\nonumber\\ &\leq\mathbb{E}\left|\int_0^t((u_{h_{\varepsilon}}^\varepsilon-u_h)\times G h_{\varepsilon}, {\rm I}_A\phi)ds\right|+
\mathbb{E}\left|\int_0^t ((u_h\times G)(h_\varepsilon-h),{\rm I}_A\phi)ds\right|\nonumber\\ &\leq C\mathbb{E}\int_0^t\|\nabla (u_{h_{\varepsilon}}^\varepsilon- u_h) \|_{L^2}^2ds+\mathbb{E}\left|\int_0^t (h_\varepsilon-h,u_h\times G^*\phi)ds\right|.
\end{align}
Next, we show that the second term on right hand side of (\ref{equ4.16}) goes to $0$, as $\varepsilon\rightarrow 0+$. Using the definition of operator $\mathcal{H}_{0}$ and (\ref{equ2.9}),
\begin{eqnarray}\label{equ4.17}
\int_0^t\|u_h\times G^*\phi\|_{\mathcal{H}^0}^2ds\leq \int_0^t\|u_h\times G^*\|_{L_Q(\mathcal{H}_0;H^1)}^2\|\phi\|_{H^1}^2ds\leq C.
\end{eqnarray}
We have $h_{\varepsilon}\rightarrow h$ weakly in $L^2(0,T;\mathcal{H}_{0})$ as $\varepsilon\rightarrow\infty$ by assumption, then
\begin{eqnarray*}
\int_0^t (h_\varepsilon-h,u_h\times G^*\phi)ds\rightarrow 0, \mathbb{P}-a.s.
\end{eqnarray*}
Therefore, by the Vitali convergence theorem, to get
\begin{eqnarray}\label{equ4.18}
\mathbb{E}\left|\int_0^t (h_\varepsilon-h,u_h\times G^*\phi)ds\right|\rightarrow 0, ~{\rm as}~\varepsilon\rightarrow\infty.
\end{eqnarray}
All the estimates (\ref{equ4.12})-(\ref{equ4.16}) and (\ref{equ4.18}) imply that
\begin{eqnarray}\label{equ4.19}
&&\mathbb{E}\bigg[(u_{h_\varepsilon}^\varepsilon,{\rm I}_A\phi)-\bigg(\nu_1\int_0^t(\triangle u_{h}, {\rm I}_A\phi)ds+\gamma\int_0^t(u_h\times\triangle u_h,{\rm I}_A\phi)ds\nonumber\\ &&\qquad-\nu_2\int_{0}^{t}((1+\mu|u_h|^2)u_h,{\rm I}_A\phi)ds+\int_0^t(u_h\times G h,{\rm I}_A\phi)ds\bigg)\bigg]\rightarrow 0.
\end{eqnarray}

On the other hand, since $u_{h_{\varepsilon}}^\varepsilon\in L^p(\Omega, L^\infty(0,T;H^1))$, we have as $\varepsilon\rightarrow 0$,
\begin{eqnarray*}
\sup_{t\in [0,T]}|(u_{h_{\varepsilon}}^\varepsilon-u_h,1_A\phi)|\rightarrow 0,~ \mathbb{P}\mbox{-a.s.}
\end{eqnarray*}
this together with the dominated convergence theorem, to deduce
\begin{eqnarray}\label{equ4.20}
\mathbb{E}\left(\sup_{t\in [0,T]}|(u_{h_{\varepsilon}}^\varepsilon-u_h,1_A\phi)|\right)\rightarrow 0.
\end{eqnarray}
We may infer from (\ref{equ4.19}) and (\ref{equ4.20}) that $u_h$ is a solution of system (\ref{equ4.2}).

By (\ref{equ2.9}), the H\"{o}lder inequality and the fact that $h\in \mathcal{A}_{M}$, we have
\begin{eqnarray}\label{equ4.21}
&&{\rm I}_{B_{N}^{\varepsilon}(T)}\sup_{t\in [0,T]}\left|\int_{0}^{t}(\nabla (u_{h}\times G)(h_{\varepsilon}-h),\nabla u_{h_{\varepsilon}}^\varepsilon-\nabla u_h)ds\right|\nonumber\\
&&\leq C\int_{0}^{T}{\rm I}_{B_{N}^{\varepsilon}(T)}\|\nabla u_{h_{\varepsilon}}^\varepsilon-\nabla u_h\|_{L^{2}}\|u_{h}\times G\|_{L_Q(\mathcal{H}_0;H^{1})}\|h_{\varepsilon}-h\|_{\mathcal{H}_{0}}dt\nonumber\\
&&\leq C\rm I}_{B_{N}^{\varepsilon}(T)}\sup_{t\in [0,T]}(1+\|u_{h}\|_{H^{1}}){\left(\int_{0}^{T}\|\nabla u_{h_{\varepsilon}}^\varepsilon-\nabla u_h\|_{L^{2}}^{2}dt\right)^{\frac{1}{2}}\left(\int_{0}^{T}\|h_{\varepsilon}-h\|_{\mathcal{H}_{0}}^2dt\right)^{\frac{1}{2}}\nonumber\\
&&\leq C(M,N)\left(\int_{0}^{T}\|\nabla u_{h_{\varepsilon}}^\varepsilon-\nabla u_h\|_{L^{2}}^{2}dt\right)^{\frac{1}{2}}\rightarrow 0,~as~\varepsilon\rightarrow 0.
\end{eqnarray}
Finally, the desired result follows from (\ref{equ4.12}).

In addition, we have by (\ref{equ2.9}) and the H\"{o}lder inequality,
\begin{eqnarray*}
&&\mathbb{E}\left({\rm I}_{B_{N}^{\varepsilon}(T)}\int_{0}^{T}((u_{h}\times G)(h_{\varepsilon}-h),\triangle V_{\varepsilon} )dt\right)^2\\
&&\leq \mathbb{E}\left({\rm I}_{B_{N}^{\varepsilon}(T)}\int_{0}^{T}\|\triangle V_\varepsilon\|_{L^2}\|u_h\times G\|_{L_Q(\mathcal{H}_0;L^2)}\|h_\varepsilon-h\|_{\mathcal{H}_0}dt\right)^2\\
&&\leq \mathbb{E}\left({\rm I}_{B_{N}^{\varepsilon}(T)}\sup_{t\in[0,T]}\|u_h\|_{L^2}\int_{0}^{T}\|\triangle V_\varepsilon\|_{L^2}\|h_\varepsilon-h\|_{\mathcal{H}_0}dt\right)^2\\
&&\leq \mathbb{E}\left({\rm I}_{B_{N}^{\varepsilon}(T)}\sup_{t\in[0,T]}\|u_h\|_{L^2}^2\int_{0}^{T}\|\triangle V_\varepsilon\|_{L^2}^2dt\int_{0}^{T}\|h_\varepsilon-h\|_{\mathcal{H}_0}^2dt\right)\\
&&\leq C(N,M),
\end{eqnarray*}
this together with  Claim 2 and the Vitali convergence theorem, we may infer
\begin{eqnarray}\label{equ4.22}
\lim_{\varepsilon\rightarrow 0}\mathbb{E}\left({\rm I}_{B_{N}^{\varepsilon}(T)}\left|\int_{0}^{T}(\nabla (u_{h}\times G)(h_{\varepsilon}-h),\nabla V_{\varepsilon} )dt\right|\right)=0.
\end{eqnarray}
Taking expectation on (\ref{equ4.9}), considering (\ref{equ4.10}), (\ref{equ4.11}) and (\ref{equ4.22}), we obtain
\begin{eqnarray}\label{equ4.23}
\mathbb E\left[{\rm I}_{B_{N}^{\varepsilon}(T)}\left(\sup_{t\in[0,T]}\|\nabla V_{\varepsilon}(t)\|_{L^{2}}^{2}
+\nu_{1}\int_{0}^{T}\|\triangle V_{\varepsilon}(s)\|_{L^{2}}^{2}dt\right)\right]\rightarrow 0,~ {\rm as }~ \varepsilon\rightarrow 0.
\end{eqnarray}
The Chebyshev inequality yields that for any $\delta>0$,
\begin{eqnarray*}
&&\mathbb P\left(\sup_{t\in[0,T]}\|\nabla V_{\varepsilon}(t)\|_{L^{2}}^{2}+\nu_{1}\int_{0}^{T}\|\triangle V_{\varepsilon}(s)\|_{L^{2}}^{2}dt>\delta\right)\\
&&\leq\mathbb P(B_{N}^{\varepsilon}(T)^{c})+\frac{1}{\delta}\mathbb{E}\left[{\rm I}_{B_{N}^{\varepsilon}(T)}\left(\sup_{t\in[0,T]}\|\nabla V_{\varepsilon}(t)\|_{L^{2}}^{2}
+\nu_{1}\int_{0}^{T}\|\triangle V_{\varepsilon}(s)\|_{L^{2}}^{2}dt\right)\right].
\end{eqnarray*}
By Claim 1, we know that for any $\delta_1>0$, there exists $N_0$, for all $N>N_0$, such that  $\mathbb{P}(B_{N}^{\varepsilon}(T)^{c})<\delta_{1}$. On the other hand, the convergence (\ref{equ4.8}) implies that for any $\varepsilon_{1}>0$ and fixed $N$, there exists $\tilde{\varepsilon}$, such that for all $\varepsilon\in [0,\tilde{\varepsilon}]$, we have
\begin{eqnarray*}
\mathbb{E}\left[{\rm I}_{B_{N}^{\varepsilon}(T)}\left(\sup_{t\in[0,T]}\|\nabla V_{\varepsilon}(t)\|_{L^{2}}^{2}+\nu_{1}\int_{0}^{T}\|\triangle V_{\varepsilon}(s)\|_{L^{2}}^{2}dt\right)\right]<\delta \varepsilon_{1}.
\end{eqnarray*}
Therefore, $\forall \delta>0$, $\mathbb P\left(\sup_{t\in[0,T]}\|\nabla V_{\varepsilon}(t)\|_{L^{2}}^{2}+\nu_{1}\int_{0}^{T}\|\triangle V_{\varepsilon}(s)\|_{L^{2}}^{2}dt>\delta\right)\rightarrow 0$ as $\varepsilon\rightarrow 0^{+}$, that is, in the sense of probability $u^{\varepsilon}_{h_{\varepsilon}}\rightarrow u_{h}$ as $\varepsilon\rightarrow 0^{+}$.

This completes the proof of Lemma \ref{lem4.1}.
\end{proof}

The following compactness result is another important factor which allow us to establish the large deviation principle for $u_{\varepsilon}$.
\begin{lemma}\label{lem4.2} For every $M<\infty$, let $K_{M}=\{u_{h}:h\in S_{M}\}$ where $u_{h}$ is the unique solution in Polish space $X$. Then, $K_{M}$ is a compact set of $X$.
\end{lemma}
\begin{proof} Let $\{u_{h_n}\}$ be a sequence in $K_{M}$ corresponding to solutions of the following system with controlled terms $\{h_{n}\}_{n\geq 1}$ in $S_{M}$:
\begin{eqnarray}\label{equ4.24}
du_{h_n}-\nu_{1}\triangle u_{h_{n}}dt=\gamma u_{h_{n}}\times \triangle u_{h_{n}}dt-\nu_{2}(1+\mu|u_{h_{n}}|^{2})u_{h_{n}}dt+(u_{h_{n}}\times G)h_{n}dt.
\end{eqnarray}
Note that the solution $u_{h_n}\in \mathcal{C}([0,T];H^1)\cap L^{2}(0,T;H^2)\cap W^{1,2}(0,T;L^2)$ is uniform bounded in $n$. Then the Aubin-Lions lemma gives that there exists a function $u_h\in L^2(0,T;H^1)$ such that $u_{h_n}\rightarrow u_h$ in $L^2(0,T;H^1)$. By the similar reason as in Lemma {\ref{lem4.1}}, we may infer that $u_h$ is the solution to system (\ref{equ4.2}) using the weak convergence of $h_n$. We now show that the subsequence of the solutions $u_{h_{n}}$ still denoted by $u_{h_n}$ converges in $X$ to $u_h$. The difference of the solutions $V_n:=u_{h_n}-u_h$ satisfies
\begin{align}\label{equ4.25}
dV_n-\nu_1 \triangle V_n dt&=\gamma(V_n\times \triangle u_{h_n}+u_h\times \triangle V_n)dt-\nu_{2}(\mu(|u_{h_n}|^2-|u_h|^2)u_{h_n}\nonumber\\
&+(1+\mu|u_h|^2)V_n)dt+((u_{h_n}\times G)h_n-(u_h\times G)h)dt.
\end{align}
Using the operator $\partial$ on both sides of (\ref{equ4.25}), then taking the inner product with $\partial V_n$, we have
\begin{eqnarray}\label{equ4.26}
&&d\|\nabla V_{n}\|_{L^{2}}^{2}+2\nu_{1}\|\triangle V_{n}\|_{L^{2}}^{2}dt=-2\gamma(V_{n}\times \triangle u_{h_{n}}+u_{h}\times \triangle V_{n}, \triangle V_{n})dt\nonumber\\
&&\qquad+2\nu_{2}\mu((|u_{h_{n}}|^{2}-|u_{h}|^{2})u_{h_{n}}, \triangle V_{n})dt+2\nu_{2}((1+\mu|u_{h}|^{2})V_{n},\triangle V_{n})dt\nonumber\\
&&\qquad-2((u_{h_n}\times G)h_n-(u_h\times G)h, \triangle V_{n})dt.
\end{eqnarray}
Integrating on interval $[0,T]$ in (\ref{equ4.26}), using the similar estimates as in (\ref{equ4.4})-(\ref{equ4.7}), we obtain
\begin{eqnarray*}
&&\sup_{t\in [0,T]}\|\nabla V_{n}\|_{L^{2}}^{2}+\nu_{1}\int_0^T\|\triangle V_{n}\|_{L^{2}}^{2}dt\nonumber\\
&&\leq C\int_{0}^{T}(1+\|\triangle u_{h_{n}}\|_{L^{2}}^{2}+\|(u_{h_{n}},u_{h})\|_{H^{1}}^{4})\|V_{n}\|_{H^{1}}^{2}dt\nonumber\\
&&\quad+\int_{0}^{T}((u_{h}\times G)(h_{n}-h),\triangle V_{n} )dt.
\end{eqnarray*}
Then, we use the Gronwall Lemma to conclude
\begin{eqnarray*}
\sup_{t\in [0,T]}\|\nabla V_{n}\|_{L^{2}}^{2}+\nu_{1}\int_0^T\|\triangle V_{n}\|_{L^{2}}^{2}dt\leq e^{C(M,T)}\left|\int_{0}^{T}((u_{h}\times G)(h_{n}-h),\triangle V_{n} )dt\right|.
\end{eqnarray*}
By the same as (\ref{equ4.21}), we have
\begin{eqnarray}\label{equ4.27}
\left|\int_{0}^{T}((u_{h}\times G)(h_{n}-h),\triangle V_{n} )dt\right|\rightarrow 0,~ {\rm as}~ n\rightarrow \infty,
\end{eqnarray}
which leads to
\begin{eqnarray*}
\sup_{t\in [0,T]}\|\nabla V_{n}\|_{L^{2}}^{2}+\nu_{1}\int_0^T\|\triangle V_{n}\|_{L^{2}}^{2}dt\rightarrow 0, ~{\rm as}~ n\rightarrow \infty.
\end{eqnarray*}
This shows that every sequence in $K_{M}$ has a convergent subsequence. Hence $K_{M}$ is a compact subset of space $X$.
\end{proof}

\noindent {\bf Proof of Theorem 2.1} Combining  Lemmas \ref{lem4.1} and \ref{lem4.2}, we get the desired result of Theorem \ref{the2.1} using Theorem \ref{the4.1}.

\section{Central limit theorem}\label{sec5}

In this section, we shall establish the central limit theorem. Since $V_\varepsilon:=\frac{u_{\varepsilon}-u_{0}}{\sqrt{\varepsilon}}$ satisfies the system
\begin{eqnarray}\label{equ5.1}
&&dV_\varepsilon-\nu_1 \triangle V_\varepsilon dt\nonumber\\
&&=\gamma(V_\varepsilon\times \triangle u_\varepsilon+u_0\times \triangle V_\varepsilon)dt-\nu_2V_\varepsilon dt-\nu_2\mu(\varepsilon |V_\varepsilon|^{2}V_\varepsilon\nonumber\\
&&\quad+(u_{0}\cdot V_\varepsilon)u_\varepsilon+(u_\varepsilon\cdot V_\varepsilon)u_\varepsilon+(u_0\cdot u_\varepsilon)V_\varepsilon)dt+u_\varepsilon\times Gd\mathcal{W},
\end{eqnarray}
where $u_0$ and $u_\varepsilon$ are the solutions of systems (\ref{equ1.1}) and (\ref{equ1.4}) respectively, and satisfy the following uniform energy estimates:
\begin{eqnarray}
&&\mathbb{E}\left(\sup_{t\in[0,T]}\|\nabla u_\varepsilon\|_{L^{2}}^{2p}\right)+\mathbb{E}\int_{0}^{T}\|\nabla u_\varepsilon\|_{L^{2}}^{2p-2}\|\triangle u_\varepsilon\|_{L^{2}}^2dt\leq C_1,\label{equ5.2}\\
&&\sup_{t\in[0,T]}\|\nabla u_0\|_{L^{2}}^{2p}+\left(\int_{0}^{T}\|\triangle u_0\|_{L^{2}}^2dt\right)^{p}\leq C_2,\; \forall p\geq 1, \label{equ5.3}
\end{eqnarray}
the constant $C_{1}$ is independent of $\varepsilon$. The energy inequality (\ref{equ5.2}) is a consequence of Lemma \ref{lem3.1} applying to $h=0$. The following lemma gives the bound uniformly in $\varepsilon$, which is the cornerstone of accomplishing the central limit theorem.

\begin{lemma}\label{lem5.1} Suppose that the operator $G$ satisfies the condition (\ref{equ1.2}), and $\tau_R$ is a stopping time defined by
\begin{eqnarray*}
\tau_R=\inf\left\{t>0:\sup_{s\in [0,t]}\|\nabla u_\varepsilon\|_{L^2}^{2p}+\left(\int_{0}^{t}\|\triangle u_\varepsilon\|_{L^{2}}^2ds\right)^{p}>R\right\}.
\end{eqnarray*}
Then, there exists a constant C which depends on $R, T, \mathcal{D}, p$ but independent of $\varepsilon$ such that
\begin{eqnarray*}
\sup_{\varepsilon\in(0,1]}\mathbb{E}\left(\sup_{s\in[0,t\wedge\tau_{R}]}\|\nabla V_\varepsilon\|_{L^{2}}^{2p}\right)+\sup_{\varepsilon\in(0,1]}\nu_1\mathbb{E}\int_{0}^{t\wedge\tau_{R}}\|\nabla V_\varepsilon\|_{L^2}^{2(p-1)}\|\triangle V_\varepsilon\|_{L^2}^2 ds\leq C,
\end{eqnarray*}
for all $t\in [0,T]$ and $p\geq 1$.
\end{lemma}
\begin{proof}Using the It\^{o} formula to function $\|\nabla V_\varepsilon\|_{L^{2}}^{2p}$,
\begin{eqnarray}\label{equ5.4}
&&d\|\nabla V_\varepsilon\|_{L^{2}}^{2p}+2p\nu_1\|\nabla V_\varepsilon\|_{L^2}^{2(p-1)}\|\triangle V_\varepsilon\|_{L^2}^2 dt\nonumber\\
&&=-2p\gamma\|\nabla V_\varepsilon\|_{L^2}^{2(p-1)}(V_\varepsilon\times \triangle u_\varepsilon+u_0\times \triangle V_\varepsilon, \triangle V_\varepsilon)dt\nonumber\\
&&\quad+2p\nu_2\|\nabla V_\varepsilon\|_{L^2}^{2(p-1)}(V_\varepsilon, \triangle V_\varepsilon)dt+2p\nu_2\mu\|\nabla V_\varepsilon\|_{L^2}^{2(p-1)}(\varepsilon |V_\varepsilon|^{2}V_\varepsilon,\triangle V_\varepsilon)dt\nonumber\\
&&\quad-2p\|\nabla V_\varepsilon\|_{L^2}^{2(p-1)}(u_\varepsilon\times Gd\mathcal{W},\triangle V_\varepsilon)\nonumber\\
&&\quad+2p\nu_2\mu\|\nabla V_\varepsilon\|_{L^2}^{2(p-1)}((u_{0}\cdot V_\varepsilon)u_\varepsilon+(u_\varepsilon\cdot V_\varepsilon)u_\varepsilon+(u_0\cdot u_\varepsilon)V_\varepsilon,\triangle V_\varepsilon)dt\nonumber\\
&&\quad+p\|\nabla V_\varepsilon\|_{L^2}^{2(p-1)}\| u_\varepsilon\times G\|_{L_Q(\mathcal{H}_0;H^1)}^{2}dt\nonumber\\
&&\quad+2p(p-1)\|\nabla V_\varepsilon\|_{L^2}^{2(p-2)}(u_\varepsilon\times G,\triangle V_\varepsilon)^2dt\nonumber\\
&&=J_1dt+\cdots+J_4d\mathcal{W}+\cdots+J_7dt.
\end{eqnarray}
By the H\"{o}lder inequality and Lemma \ref{lem2.1} (\ref{equ2.4}), we have
\begin{align}\label{equ5.5}
|J_1|&\leq 2p\gamma\|\nabla V_\varepsilon\|_{L^2}^{2(p-1)}\|\triangle u_\varepsilon\|_{L^2}\|\triangle V_\varepsilon\|_{L^2}\|V_\varepsilon\|_{H^1}\nonumber\\
&\leq \frac{p\nu_1}{2}\|\nabla V_\varepsilon\|_{L^2}^{2(p-1)}\|\triangle V_\varepsilon\|_{L^2}^2+C\|\triangle u_\varepsilon\|_{L^2}^2\|V_\varepsilon\|_{H^1}^{2p}.
\end{align}
Regarding the term $J_3$, we have
\begin{eqnarray}\label{equ5.6}
J_3=-2p\nu_2\varepsilon\|\nabla V_\varepsilon\|_{L^2}^{2(p-1)}(3\mu(|V_\varepsilon|^2,|\nabla V_\varepsilon|^2)+\|\nabla V_\varepsilon\|_{L^2}^2).
\end{eqnarray}
The term $J_{5}$ may be handled as,
\begin{align}\label{equ5.7}
|J_5|&\leq 2p\mu\varepsilon\|\nabla V_\varepsilon\|_{L^2}^{2(p-1)}\|\triangle V_\varepsilon\|_{L^2}\|V_\varepsilon\|_{L^2}\|(u_0, u_\varepsilon)\|_{L^\infty}^2\nonumber\\
&\leq \frac{p\nu_1}{2}\|\nabla V_\varepsilon\|_{L^2}^{2(p-1)}\|\triangle V_\varepsilon\|_{L^2}^2+C(\varepsilon)\|(u_0, u_\varepsilon)\|_{H^1}^4\|V_\varepsilon\|_{H^1}^{2p}.
\end{align}
The stochastic term, which may be treated by same way as in (\ref{equ3.7}), to yield
\begin{eqnarray}\label{equ5.8}
&&\mathbb{E}\left(\sup_{s\in[0,t\wedge \tau_{R}]}\left|\int_{0}^{s}\|\nabla V_\varepsilon\|_{L^2}^{2(p-1)}(u_\varepsilon\times G,\triangle V_\varepsilon)d\mathcal{W}\right|\right)\nonumber\\
&&\leq C\mathbb{E}\left(\int_{0}^{t\wedge \tau_{R}}\|\nabla V_\varepsilon\|_{L^2}^{4(p-1)}\|\nabla V_\varepsilon\|_{L^2}^{2}\|
u_\varepsilon\times G\|_{L_Q(\mathcal{H}_0;H^1)}^{2}dt\right)^\frac{1}{2}\nonumber\\
&&\leq C\mathbb{E}\left(\sup_{s\in[0,t\wedge \tau_{R}]}\|\nabla V_\varepsilon\|_{L^2}^{p}\right)\left(\int^{t\wedge \tau_{R}}_0\|\nabla V_\varepsilon\|_{L^2}^{2(p-1)}\|\nabla u_\varepsilon\|_{L^2}^2ds\right)^\frac{1}{2}\nonumber\\
&&\leq \frac{1}{2}\mathbb{E}\left(\sup_{t\in[0,t\wedge \tau_{R}]}\|\nabla V_\varepsilon\|_{L^2}^{2p}\right)+C\mathbb{E}\int^{t\wedge \tau_{R}}_0\|\nabla V_\varepsilon\|_{L^2}^{2p}ds\nonumber\\
&&\quad+C\mathbb{E}\int^{t\wedge \tau_{R}}_0\|\nabla u_\varepsilon\|_{L^2}^{2p}ds.
\end{eqnarray}
Also, similarly to estimate (\ref{equ3.11}),  we have
\begin{eqnarray}\label{equ5.9}
|J_6|\leq C\|\nabla V_\varepsilon\|_{L^2}^{2(p-1)}\|\nabla u_\varepsilon\|_{L_2}^{2}\leq C\|\nabla V_\varepsilon\|_{L^2}^{2p}+C\|\nabla u_\varepsilon\|_{L_2}^{2p},
\end{eqnarray}
as well as
\begin{eqnarray}\label{equ5.10}
|J_7|\leq C\|\nabla V_\varepsilon\|_{L^2}^{2(p-2)}\|\nabla V_\varepsilon\|_{L^2}^2\|u_\varepsilon\times G\|_{L_Q(\mathcal{H}_0;H^1)}^{2}\leq C\|\nabla V_\varepsilon\|_{L^2}^{2p}+C\|\nabla u_\varepsilon\|_{L^2}^{2p}.
\end{eqnarray}
Define the stopping time $\tau=\tau_{N}\wedge \tau_{R}$, where
\begin{eqnarray*}
\tau_{N}:=\inf\left\{t>0;\sup_{s\in [0,t]}\|\nabla V_\varepsilon\|_{L^2}^{2p}\geq N\right\}.
\end{eqnarray*}
Taking the supremum on interval $[0,t\wedge\tau]$ in (\ref{equ5.4}), and then taking expectation, combining the estimates (\ref{equ5.5})-(\ref{equ5.10}), to conclude
\begin{eqnarray*}
&&\mathbb{E}\left(\sup_{s\in[0,t\wedge\tau]}\|\nabla V_\varepsilon\|_{L^{2}}^{2p}\right)+\mathbb{E}\int_{0}^{t\wedge\tau}\nu_1\|\nabla V_\varepsilon\|_{L^2}^{2(p-1)}\|\triangle V_\varepsilon\|_{L^2}^2 ds\nonumber\\
&&\leq C\mathbb{E}\int_{0}^{t\wedge\tau}\|\nabla V_\varepsilon\|_{L^{2}}^{2p}ds+C\mathbb{E}\int_{0}^{t\wedge\tau}(1+\|\nabla u_\varepsilon\|_{L_2}^2)^{p}ds\nonumber\\
&&\quad+\mathbb{E}\int_{0}^{t\wedge\tau}(\|\triangle u_\varepsilon\|_{L^2}^2+\|u_0, u_\varepsilon\|_{H^1}^4)\|\nabla V_\varepsilon\|_{L^{2}}^{2p}ds.
\end{eqnarray*}
The Gronwall Lemma implies that
\begin{eqnarray*}
&&\mathbb{E}\left(\sup_{s\in[0,t\wedge\tau]}\|\nabla V_\varepsilon\|_{L^{2}}^{2p}\right)+\nu_1\mathbb{E}\int_{0}^{t\wedge\tau}\|\nabla V_\varepsilon\|_{L^2}^{2(p-1)}\|\triangle V_\varepsilon\|_{L^2}^2 ds\nonumber\\&& \leq e^{C(R,T,p)}\mathbb{E}\int_{0}^{t\wedge\tau}(1+\|\nabla u_\varepsilon\|_{L^2}^2)^{p}ds\leq C(R,T,p),
\end{eqnarray*}
where the constant $C$ is independent of $\varepsilon$. Letting $N\rightarrow \infty$, we get the desired result,
\begin{eqnarray*}
&&\sup_{\varepsilon\in(0,1]}\mathbb{E}\left(\sup_{s\in[0,t\wedge\tau_{R}]}\|\nabla V_\varepsilon\|_{L^{2}}^{2p}\right)+\sup_{\varepsilon\in(0,1]}\nu_1\mathbb{E}\int_{0}^{t\wedge\tau_{R}}\|\nabla V_\varepsilon\|_{L^2}^{2(p-1)}\|\triangle V_\varepsilon\|_{L^2}^2 ds\nonumber\\&&\leq C(R,T,p).
\end{eqnarray*}
This completes the proof of Lemma \ref{lem5.1}.
\end{proof}

We shall show that $V_\varepsilon$ converges to the solution of the system:
\begin{eqnarray}\label{equ5.11}
&&dV_0-\nu_1\triangle V_0dt=\gamma(V_0\times \triangle u_0+u_0\times \triangle V_0)dt-\nu_2V_0 dt\nonumber\\
&&\qquad\qquad\qquad\quad-\nu_2\mu(2(u_{0}\cdot V_0)u_0+|u_0|^2V_0)dt+u_0\times Gd\mathcal{W},
\end{eqnarray}
with the initial data $V_0(0)=0$. Before that, we give the well-posedness of solution to system (\ref{equ5.11}) by the following proposition.
\begin{proposition}\label{pro5.1} Let $(\Omega, \mathcal{F}, \mathbb{P})$ be a fixed probability space. Suppose that $u_0$ is a strong solution of system (\ref{equ1.1}) and the operator $G$ satisfies Condition (\ref{equ1.2}). Then, there exists a unique solution $V_0$ to system (\ref{equ5.11}) in the following sense: the process
\begin{eqnarray*}
V_0\in L^{\infty}(0,T;H^1)\cap L^2(0,T;H^2),~\mathbb{P}\mbox{-a.s.}
\end{eqnarray*}
and for $t\in [0,T]$, it holds that $\mathbb{P}$-a.s.
\begin{eqnarray*}
&&V_0(t)-\nu_1\int_{0}^{t}\triangle V_0ds=\gamma\int_{0}^{t}(V_0\times \triangle u_0+u_0\times \triangle V_0)ds\nonumber\\
&&\qquad-\nu_2\int_{0}^{t}(2\mu(u_{0}\cdot V_0)u_0+(1+\mu|u_0|^2)V_0)ds+\int_{0}^{t}u_0\times Gd\mathcal{W}.
\end{eqnarray*}
Moreover, the energy estimate holds
\begin{eqnarray}\label{equ5.12}
\mathbb{E}\left(\sup_{t\in[0,T]}\|\nabla V_0\|_{L^2}^2\right)+\mathbb{E}\int_0^T\|\triangle V_0\|_{L^2}^2dt\leq C.
\end{eqnarray}
\end{proposition}
\begin{proof}The proof of the existence of solution to the stochastic system (\ref{equ5.11}) may be achievable by the classical pathwise argument method. Actually, this argument is easier comparing with the original system (\ref{equ1.1}) driven by additive noise, due to the fact that there is no nonlinear term appearing on the right hand of this system. Here, we just give a brief proof of the well-posedness, the rigorous proof relies on the Galerkin approximation and the procedure of passing limit. Similar result can be found in \cite{Zabczyk} for the Navier-Stokes equation.

Considering the auxiliary process $U$ which is the solution of the system:
\begin{eqnarray*}
dU+\nu_1 \triangle Udt=u_0\times Gd\mathcal{W},
\end{eqnarray*}
with the initial data $U(0)=0$. We know that the solution $U$ is $H^1$-valued stationary process with continuous trajectories, see \cite{Zabczyk}. Moreover, there exists a constant $C$ such that
\begin{eqnarray}\label{equ5.13}
\mathbb{E}\left(\sup_{t\in [0,T]}\|U\|_{H^1}^2\right)+\nu_1\mathbb{E}\int_{0}^{T}\|\triangle U\|_{L^2}^2dt\leq C.
\end{eqnarray}
Let $V=V_0-U$, then $V$ satisfies the following equation,
\begin{eqnarray}\label{equ5.14}
&&dV-\nu_1\triangle V dt=\gamma(V\times \triangle u_0+U\times \triangle u_0+u_0\times \triangle V+u_0\times \triangle U)dt\nonumber\\
&&\qquad\qquad\qquad\quad-\nu_2(2\mu(u_{0}\cdot (V+U))u_0+(1+\mu|u_0|^2)(V+U))dt.
\end{eqnarray}
Taking the inner product with $-\triangle V$ on both sides of (\ref{equ5.14}), we have
 \begin{eqnarray*}
&&d\|\nabla V\|_{L^2}^2+2\nu_1\|\triangle V\|_{L^2}^2 dt\nonumber\\
&&=-2\gamma(V\times \triangle u_0+U\times \triangle u_0+u_0\times \triangle V+u_0\times \triangle U,\triangle V)dt\nonumber\\
&&\quad+\nu_2(2\mu(u_{0}\cdot V)u_0+2\mu(u_{0}\cdot U)u_0+(1+\mu|u_0|^2)(V+U),\triangle V)dt.
\end{eqnarray*}
Using Lemma \ref{lem2.1}(\ref{equ2.4}), we obtain
\begin{eqnarray}
&&|\gamma(V\times \triangle u_0+U\times \triangle u_0, \triangle V)|\nonumber\\
&&\leq \frac{1}{4}\|\triangle V\|_{L^2}^2+C\|\triangle u_0\|_{L^2}^2(\|V\|_{H^1}^2+\|U\|_{H^1}^2),\quad \label{equ5.15}\\
&&|\gamma(u_0\times \triangle U,\triangle V)\leq \frac{1}{4}\|\triangle V\|_{L^2}^2+C\|\triangle U\|_{L^2}^2\|u_0\|_{H^1}^2. \label{equ5.16}
\end{eqnarray}
By the H\"{o}lder equality and the interpolation inequality (\ref{equ2.2}),
\begin{eqnarray}\label{equ5.17}
&&|(\mu(u_{0}\cdot (V+U))u_0+\mu|u_0|^2(V+U),\triangle V)|\nonumber\\
&&\leq \frac{1}{4}\|\triangle V\|_{L^2}^2+C(\|V\|_{H^1}^2+\|U\|_{H^1}^2)\|u_0\|_{L^4}^4.
\end{eqnarray}
From above estimates (\ref{equ5.16})-(\ref{equ5.17}), we get
\begin{align*}
d\|\nabla V\|_{L^2}^2+\nu_1\|\triangle V\|_{L^2}^2 dt&\leq C\|\triangle U\|_{L^2}^2\|u_0\|_{H^1}^2+C(\|\triangle u_0\|_{L^2}^2+\|u_0\|_{L^4}^4)\|V\|_{H^1}^2\nonumber\\
&\quad+C\|U\|_{H^1}^2(\|u_0\|_{L^4}^4+\|\triangle u_0\|_{L^2}^2).
\end{align*}
Therefore, for every $w\in \Omega$, using energy estimates (\ref{equ5.13}) and (\ref{equ5.3}), the Gronwall Lemma yields
 \begin{eqnarray*}
\sup_{t\in [0,T]}\|\nabla V\|_{L^2}^2+\nu_1\int_{0}^{T}\|\triangle V\|_{L^2}^2 dt\leq C.
\end{eqnarray*}
Since $V=V_0-U$, we infer from the properties of $U$ that $\mathbb{P}$-a.s.
\begin{eqnarray*}
V_0\in L^\infty(0,T;H^1)\cap L^2(0,T;H^2).
\end{eqnarray*}
The proof of (\ref{equ5.12}) may be obtained by applying the It\^{o} formula, stopping time, and the estimates as in Lemma \ref{lem3.1}. Here we have to be aware of is that all the calculation should be performed on the Galerkin approximate solution, and then the energy inequality (\ref{equ5.12}) is a consequence of lower-continuity of norm.

Since the detail of proving the uniqueness is similar to the argument in Lemma \ref{lem3.1}, so we omit it. This completes the proof of Proposition \ref{pro5.1}.
\end{proof}

We now have all to give the result of central limit theorem.
\begin{proof}[\bf{Proof of Theorem 2.2}] The difference $V_\varepsilon-V_0$ satisfies
\begin{eqnarray*}
&&d(V_\varepsilon-V_0)-\nu_1\triangle (V_\varepsilon-V_0)dt\nonumber\\
&&=\gamma(V_\varepsilon\times \triangle u_\varepsilon+u_0\times \triangle V_\varepsilon-V_0\times \triangle u_0-u_0\times \triangle V_0)dt\nonumber\\
&&\quad-\nu_2(V_\varepsilon-V_0)dt-\varepsilon\nu_2\mu |V_\varepsilon|^{2}V_\varepsilon dt+(u_\varepsilon-u_0)\times Gd\mathcal{W}\nonumber\\
&&\quad-\nu_2\mu\big[(u_{0}\cdot V_\varepsilon)u_\varepsilon-(u_{0}\cdot V_0)u_0+(u_\varepsilon\cdot V_\varepsilon)u_\varepsilon-(u_0\cdot V_0)u_0\nonumber\\
&&\quad+(u_0\cdot u_\varepsilon)V_\varepsilon-(u_0\cdot u_0)V_0\big]dt.
\end{eqnarray*}
Using the It\^{o} formula to function $\|\nabla (V_\varepsilon-V_0)\|_{L^{2}}^{2}$, we have
\begin{eqnarray}\label{equ5.18}
&&d\|\nabla (V_\varepsilon-V_0)\|_{L^2}^2+2\nu_1\|\triangle (V_\varepsilon-V_0)\|_{L^2}^2dt\nonumber\\
&&=-2\gamma((V_\varepsilon-V_0)\times \triangle u_\varepsilon+V_0\times\triangle(u_\varepsilon-u_0)+u_0\times \triangle (V_\varepsilon-V_0),\triangle(V_\varepsilon-V_0))dt\nonumber\\
&&\quad-2\nu_2\|\nabla (V_\varepsilon-V_0)\|_{L^2}^2dt+2(\varepsilon\nu_2\mu |V_\varepsilon|^{2}V_\varepsilon,\triangle(V_\varepsilon-V_0))dt\nonumber\\
&&\quad+2((u_\varepsilon-u_0)\times G,\triangle(V_\varepsilon-V_0))d\mathcal{W}+\|(u_\varepsilon-u_0)\times G\|_{L_Q(\mathcal{H}_0;H^1)}^2dt\nonumber\\
&&\quad+2\nu_2\mu((u_{0}\cdot V_\varepsilon)u_\varepsilon-(u_{0}\cdot V_0)u_0,\triangle(V_\varepsilon-V_0))dt\nonumber\\
&&\quad+2\nu_2\mu((u_\varepsilon\cdot V_\varepsilon)u_\varepsilon-(u_0\cdot V_0)u_0,\triangle(V_\varepsilon-V_0))dt\nonumber\\
&&\quad+2\nu_2\mu((u_0\cdot u_\varepsilon)V_\varepsilon-(u_0\cdot u_0)V_0,\triangle(V_\varepsilon-V_0))dt\nonumber\\
&&=J_1+\cdots+J_{4}d\mathcal{W}+\cdots+J_8.
\end{eqnarray}
We next estimate the above terms in turn. By Lemma \ref{lem2.1}, we have
\begin{eqnarray}\label{equ5.19}
&&|((V_\varepsilon-V_0)\times \triangle u_\varepsilon, \triangle(V_\varepsilon-V_0))|\nonumber\\ &&\leq \frac{\nu_1}{8}\|\triangle(V_\varepsilon-V_0)\|_{L^2}^2+C\|\triangle u_\varepsilon\|_{L^2}^2\|V_\varepsilon-V_0\|_{H^1}^2,
\end{eqnarray}
and
\begin{eqnarray}\label{equ5.20}
(V_0\times\triangle(u_\varepsilon-u_0)+u_0\times \triangle (V_\varepsilon-V_0),\triangle(V_\varepsilon-V_0))=0.
\end{eqnarray}
By the H\"{o}lder inequality,
\begin{eqnarray}\label{equ5.21}
|(\varepsilon\nu_2\mu |V_\varepsilon|^{2}V_\varepsilon,\triangle(V_\varepsilon-V_0))|\leq \frac{\nu_1}{8}\|\triangle(V_\varepsilon-V_0)\|_{L^2}^2+C\varepsilon\|V_\varepsilon\|_{L^6}^3.
\end{eqnarray}
For terms $J_6, J_7$ and $J_8$, by the interpolation inequality (\ref{equ2.2}) and H\"{o}lder inequality, we get
\begin{eqnarray}\label{equ5.22}
&&|\nu_2\mu((u_{0}\cdot V_\varepsilon)u_\varepsilon-(u_{0}\cdot V_0)u_0,\triangle(V_\varepsilon-V_0))|\nonumber\\
&&=|\sqrt{\varepsilon}(u_0\cdot V_\varepsilon)V_\varepsilon+(u_0\cdot(V_\varepsilon-V_0))u_0,\triangle(V_\varepsilon-V_0)) |\nonumber\\
&&\leq \frac{\nu_1}{8}\|\triangle(V_\varepsilon-V_0)\|_{L^2}^2+C\varepsilon\|u_0\|_{H^1}^2\|V_\varepsilon\|_{L^4}^4+C\|u_0\|_{H^1}^4\|V_\varepsilon-V_0\|_{H^1}^2,
\end{eqnarray}
and
\begin{eqnarray}\label{equ5.23}
&&|((u_{0}\cdot V_\varepsilon)u_\varepsilon-(u_{0}\cdot V_0)u_0,\triangle(V_\varepsilon-V_0))|\nonumber\\
&&=|(\sqrt{\varepsilon}(u_\varepsilon\cdot V_\varepsilon)V_\varepsilon+\sqrt{\varepsilon}|V_\varepsilon|^2u_0+[u_0\cdot(V_\varepsilon-V_0)]u_0, \triangle(V_\varepsilon-V_0))|\nonumber\\
&&\leq\frac{\nu_1}{8}\|\triangle(V_\varepsilon-V_0)\|_{L^2}^2+C\varepsilon\|V_\varepsilon\|_{H^1}^2(\|u_{\varepsilon}\|_{L^2}^2+\|u_{0}\|_{L^2}^2)\nonumber\\
&&\quad+C\|u_0\|_{L^4}^4\|V_\varepsilon-V_0\|_{H^{1}}^2,
\end{eqnarray}
and
\begin{eqnarray}\label{equ5.24}
|J_8|\leq \frac{\nu_1}{4}\|\triangle(V_\varepsilon-V_0)\|_{L^2}^2+C\varepsilon\|u_0\|_{H^1}^2\|V_\varepsilon\|_{L^4}^4+C\|u_0\|_{H^1}^4\|V_\varepsilon-V_0\|_{H^1}^2.
\end{eqnarray}
Define the stopping time $\tau=\tau_{N}\wedge \tau_{R}$, where the stopping time $\tau_{R}$ is that one in Lemma \ref{lem5.1}, and
$\tau_{N}:=\inf\Big\{t>0;\sup_{s\in[0,t]}\|\nabla(V_\varepsilon-V_0)\|_{L^2}^{2}\geq N\Big\}.$
For the stochastic term, we have by the Burkholder-Davis-Gundy inequality once more
\begin{eqnarray}\label{equ5.25}
&&\mathbb{E}\left(\sup_{s\in[0,t\wedge \tau]}\left|\int_{0}^{s}((u_\varepsilon-u_0)\times G,\triangle(V_\varepsilon-V_0))d\mathcal{W}\right|\right)\nonumber\\
&&\leq C\mathbb{E}\left(\int_{0}^{t\wedge \tau}\|\nabla(V_\varepsilon-V_0)\|_{L^2}^{2}\|(u_\varepsilon-u_0)\times G\|_{L_Q(\mathcal{H}_0;H^1)}^2ds\right)^\frac{1}{2}\nonumber\\
&&\leq C\sqrt{\varepsilon} \mathbb{E}\left(\int_{0}^{t\wedge \tau}\|\nabla(V_\varepsilon-V_0)\|_{L^2}^{2}\|\nabla V_\varepsilon\|_{L^2}^2ds\right)^\frac{1}{2}\nonumber\\
&&\leq \frac{1}{4}\mathbb{E}\left(\sup_{s\in[0,t\wedge \tau]}\|\nabla(V_\varepsilon-V_0)\|_{L^2}^{2}\right)+\varepsilon C\mathbb{E}\int_{0}^{t\wedge \tau}\|\nabla V_\varepsilon\|_{L^2}^2ds.
\end{eqnarray}
Taking the supremum on interval $[0,t\wedge\tau]$, and then taking expectation on both sides of (\ref{equ5.18}), yields
\begin{eqnarray}\label{equ5.26}
&&\mathbb{E}\left(\sup_{s\in [0,t\wedge\tau]}\|\nabla (V_\varepsilon-V_0)\|_{L^2}^2\right)+2\nu_1\mathbb{E}\int_{0}^{t\wedge\tau}\|\triangle (V_\varepsilon-V_0)\|_{L^2}^2ds\nonumber\\
&&=-2\gamma\mathbb{E}\int_{0}^{t\wedge\tau}((V_\varepsilon-V_0)\times \triangle u_\varepsilon+V_0\times\triangle(u_\varepsilon-u_0),\triangle(V_\varepsilon-V_0))ds\nonumber\\
&&\quad-2\gamma\mathbb{E}\int_{0}^{t\wedge\tau}((u_0\times \triangle (V_\varepsilon-V_0),\triangle(V_\varepsilon-V_0))ds\nonumber\\
&&\quad-2\mathbb{E}\int_{0}^{t\wedge\tau}\mu\|\nabla (V_\varepsilon-V_0)\|_{L^2}^2-\varepsilon \nu_2\mu ( |V_\varepsilon|^{2}V_\varepsilon,\triangle(V_\varepsilon-V_0))ds\nonumber\\
&&\quad+2\mathbb{E}\left(\sup_{s\in[0,t\wedge\tau]}\int_{0}^{s}((u_\varepsilon-u_0)\times G,\triangle(V_\varepsilon-V_0))d\mathcal{W}\right)\nonumber\\
&&\quad+\mathbb{E}\int_{0}^{t\wedge\tau}\|(u_\varepsilon-u_0)\times G\|_{L_Q(\mathcal{H}_0;H^1)}^2ds\nonumber\\
&&\quad+2\nu_2\mu\mathbb{E}\int_{0}^{t\wedge\tau}((u_{0}\cdot V_\varepsilon)u_\varepsilon-(u_{0}\cdot V_0)u_0,\triangle(V_\varepsilon-V_0))ds\nonumber\\
&&\quad+2\nu_2\mu\mathbb{E}\int_{0}^{t\wedge\tau}((u_\varepsilon\cdot V_\varepsilon)u_\varepsilon-(u_0\cdot V_0)u_0,\triangle(V_\varepsilon-V_0))ds\nonumber\\
&&\quad+2\nu_2\mu\mathbb{E}\int_{0}^{t\wedge\tau}((u_0\cdot u_\varepsilon)V_\varepsilon-(u_0\cdot u_0)V_0,\triangle(V_\varepsilon-V_0))ds.
\end{eqnarray}
Taking into account of (\ref{equ5.19})-(\ref{equ5.26}), we obtain
\begin{eqnarray*}
&&\mathbb{E}\left(\sup_{s\in [0,t\wedge\tau]}\|\nabla (V_\varepsilon-V_0)\|_{L^2}^2\right)+\nu_1\mathbb{E}\int_{0}^{t\wedge\tau}\|\triangle (V_\varepsilon-V_0)\|_{L^2}^2ds\nonumber\\
&&\leq \mathbb{E}\int_{0}^{t\wedge\tau}(\|\triangle u_\varepsilon\|_{L^2}^2+\|u_0\|_{H^1}^4)\|V_\varepsilon-V_0\|_{H^1}^2ds\nonumber\\
&&+C\varepsilon\mathbb{E}\int_{0}^{t\wedge\tau}(1+\|u_0\|_{H^1}^2+\|u_{\varepsilon}\|_{L^2}^2)\|V_\varepsilon\|_{H^1}^4ds.
\end{eqnarray*}
By the Gronwall Lemma,
\begin{eqnarray}\label{equ5.27}
&&\mathbb{E}\left(\sup_{s\in [0,t\wedge\tau]}\|\nabla (V_\varepsilon-V_0)\|_{L^2}^2\right)+\nu_1\mathbb{E}\int_{0}^{t\wedge\tau}\|\triangle (V_\varepsilon-V_0)\|_{L^2}^2ds\nonumber\\
&&\leq Ce^{C(R,p,T)}\varepsilon\mathbb{E}\int_{0}^{t\wedge\tau}(1+\|u_0\|_{H^1}^2+\|u_{\varepsilon}\|_{L^2}^2)\|V_\varepsilon\|_{H^1}^4ds\nonumber\\
&&\leq \varepsilon C(R,p,T)e^{C(R,p,T)}.
\end{eqnarray}
For any fixed $R$, let $\varepsilon\rightarrow 0$ in (\ref{equ5.27}),
\begin{eqnarray*}
\lim_{\varepsilon\rightarrow 0}\left[\mathbb{E}\Big(\sup_{s\in [0,t\wedge\tau]}\|\nabla (V_\varepsilon-V_0)\|_{L^2}^2\Big)+\nu_1\mathbb{E}\int_{0}^{t\wedge\tau}\|\triangle (V_\varepsilon-V_0)\|_{L^2}^2ds\right]=0.
\end{eqnarray*}
Since $\tau_{R}\rightarrow T$ as $R\rightarrow \infty$, by the dominated convergence theorem, we have for any $t\in [0,T]$ as $R,N\rightarrow\infty$,
\begin{eqnarray*}
\lim_{\varepsilon\rightarrow 0}\left[\mathbb{E}\Big(\sup_{s\in [0,t]}\|\nabla (V_\varepsilon-V_0)\|_{L^2}^2\Big)+\nu_1\mathbb{E}\int_{0}^{t}\|\triangle (V_\varepsilon-V_0)\|_{L^2}^2ds\right]=0.
\end{eqnarray*}
This completes the proof of Theorem \ref{the2.2}.
\end{proof}

\section*{Acknowledgments}
Z.Qiu is supported by the CSC under grant No.201806160015. Y.Tang is supported by the National Natural Science Foundation of China under Grant No.11971188. H.Wang is supported by National Postdoctoral Program for Innovative Talents (No. BX201600020) and Project No. 2019CDXYST0015 supported by the Fundamental Research Funds for the Central Universities.

\end{document}